\documentclass[10pt,a4paper]{article}
\usepackage{xcolor}
\usepackage{mdframed}
\usepackage[utf8]{inputenc}
\usepackage{amsmath, amsthm, amssymb, bbm, bm}
\usepackage{tikz-cd}
\usepackage{enumerate}
\usepackage{appendix}
\usepackage[hidelinks]{hyperref}
\usepackage{color}

\definecolor{mygray}{gray}{0.75}

\newtheorem{tvrz}{Proposition}[section]

\newtheorem{theorem}[tvrz]{Theorem}

\theoremstyle{definition}
\newtheorem{definice}[tvrz]{Definition}
\theoremstyle{remark}
\newtheorem{rem}[tvrz]{Remark}
\theoremstyle{definition}
\newtheorem{mdexample}[tvrz]{Example}
\newenvironment{example}%
{\begin{mdframed}[topline=false, rightline=false, bottomline=false, linewidth=0.2em, linecolor=mygray, innerleftmargin=0.5em, innerrightmargin=0,leftmargin=-0.7em]\begin{mdexample}}%
{\end{mdexample}\end{mdframed}}

\def\^{\wedge}
\def\<{\langle}
\def\>{\rangle}
\def\M{\mathcal{M}}

\def\N{\mathbb{N}}
\def\Cn{\mathbb{C}}
\def\cN{\mathcal{N}}

\def\cL{\mathcal{L}}

\def\X{\mathfrak{X}}
\def\g{\mathfrak{g}}
\def\d{\mathfrak{d}}

\def\R{\mathbb{R}}
\def\C{\mathcal{C}}
\def\dD{\mathrm{D}}

\def\E{\mathcal{E}}

\def\J{\mathcal{J}}

\def\Z{\mathbb{Z}}

\def\ol{\overline}

\def\fJ{\mathbf{J}}

\def\fp{\mathbf{p}}
\def\fq{\mathbf{q}}

\def\frm{\mathfrak{m}}
\def\frc{\mathfrak{c}}

\def\bbz{\mathbbm{z}}

\def\dr{\mathrm{d}}
\def\da{\mathsf{a}}

\def\1{\mathbbm{1}}

\def\f1{\mathbf{1}}

\def\fJ{\mathbf{J}}
\def\A{\mathcal{A}}
\def\~{\widetilde}

\def\tr{\triangleright}

\newcommand{\Li}[1]{\mathcal{L}_{#1}}
\newcommand\ul[1]{\underline{#1}}

\DeclareMathOperator{\gdim}{gdim}

\DeclareMathOperator{\Op}{\textbf{Op}}

\DeclareMathOperator{\gVect}{\textbf{gVec}}

\DeclareMathOperator{\gcAs}{\textbf{gcAs}}

\DeclareMathOperator{\ad}{ad}

\DeclareMathOperator{\Lin}{Lin}

\DeclareMathOperator{\gDer}{gDer}

\DeclareMathOperator{\an}{an}

\DeclareMathOperator{\gr}{gr}

\DeclareMathOperator{\grk}{grk}

\begin{document}
\begin{flushright}
\today
\end{flushright}
\vspace{0.7cm}
\begin{center}

\baselineskip=13pt {\Large \bf{Graded Generalized Geometry}\\}
 \vskip0.5cm
 {\it Dedicated to Branislav Jurčo on the occasion of his 60th birthday}  
 \vskip0.7cm
 {\large{Jan Vysoký$^{1}$}}\\
 \vskip0.6cm
$^{1}$\textit{Faculty of Nuclear Sciences and Physical Engineering, Czech Technical University in Prague\\ Břehová 7, 115 19 Prague 1, Czech Republic, jan.vysoky@fjfi.cvut.cz}\\
\vskip0.3cm
\end{center}

\begin{abstract}
Generalized geometry finds many applications in the mathematical description of some aspects of string theory. In a nutshell, it explores various structures on a generalized tangent bundle associated to a given manifold. In particular, several integrability conditions can be formulated in terms of a canonical Dorfman bracket, an example of Courant algebroid. On the other hand, smooth manifolds can be generalized to involve functions of $\Z$-graded variables which do not necessarily commute. This leads to a mathematical theory of graded manifolds. It is only natural to combine the two theories by exploring the structures on a generalized tangent bundle associated to a given graded manifold. 

After recalling elementary graded geometry, graded Courant algebroids on graded vector bundles are introduced. We show that there is a canonical bracket on a generalized tangent bundle associated to a graded manifold. Graded analogues of Dirac structures and generalized complex structures are explored. We introduce differential graded Courant algebroids which can be viewed as a generalization of Q-manifolds. A definition and examples of graded Lie bialgebroids are given.  
\end{abstract}


\section*{Introduction} \label{sec_introduction}
\addcontentsline{toc}{section}{Introduction}
Generalized geometry \cite{Hitchin:2004ut, Gualtieri:2003dx, cavalcanti2005new} and Courant algebroids \cite{liu1997manin, 1999math.....10078R, Severa:2017oew} became one of the most useful tools in the understanding of geometry behind string theory. Courant algebroids can be used to describe current algebras of $\sigma$-models \cite{Alekseev:2004np}, various aspects of T-duality \cite{Baraglia:2013wua, Cavalcanti:2011wu, Garcia-Fernandez:2016ofz, Grana:2008yw}, geometrical approach to (exceptional, heterotic) supergravity \cite{Coimbra:2011nw, Coimbra:2012af, garcia2014torsion, Strickland-Constable:2019qti} and Poisson--Lie T-duality \cite{Severa:2015hta,Severa:2016prq, Severa:2017kcs, Severa:2018pag}. See also our contributions to the subject \cite{jurvco2016heterotic,Jurco:2016emw, Vysoky:2017epf, Jurco:2017gii, Jurco:2019tgt}.

On the other and, graded manifolds rose became a useful tool both in differential geometry and mathematical physics. They can be viewed as a modification of supergeometry, allowing for a grading of local coordinates governed by the abelian group $\mathbb{Z}$. Theory of graded manifolds is a crucial mathematical notion for BV \cite{batalin1983quantization,batalin1984gauge} and AKSZ \cite{alexandrov1997geometry} formalism in quantum field theory. There are excellent materials covering this topic, see e.g. \cite{qiu2011introduction,ikeda2017lectures,2011RvMaP..23..669C}. Graded manifolds provide an elegant description of Courant algebroids \cite{roytenberg2002structure}, they can be used to construct Drinfel'd doubles for Lie bialgebroids \cite{Voronov:2001qf}, Lie algebroids and their higher analogues \cite{Voronov:2010hj} or they are utilized for integration of Courant algebroids \cite{li2011integration, vsevera2015integration}. In our treatment of graded geometry, we follow the definitions and formalism established in \cite{vysoky2021global}. Let us point out that recently a similar treatment appeared in \cite{2021arXiv210813496K}. See also \cite{mehta2006supergroupoids, jubin2019differential, fairon2017introduction}. 

This paper aims to develop a graded version of generalized geometry. In particular, we will study structures on the direct sum $T[\ell]\M \oplus T^{\ast}\M$, where $\M$ is a general graded manifold and $T[\ell]\M$ and $T^{\ast}\M$ are its (degree shifted) tangent and cotangent bundles. More generally, we are going to introduce a concept of a graded Courant algebroid and provide non-trivial examples. In this way, we obtain a nice geometrical realization of (graded) Loday brackets introduced in \cite{Kosmann1996}. 

The paper is organized as follows:

In \S 1, we summarize a basic knowledge of graded geometry necessary for this paper. In particular, we give a definition of a graded manifold and a basic example of a graded domain. Graded vector bundles are introduced together with an example of a tangent bundle to a graded manifold. We recall basic notions of graded vector bundle theory, e.g. dual graded vector bundles, degree shifted graded vector bundles, and vector bundle maps. We introduce a graded version of an exterior algebra of differential forms. Rather then taking a usual path using local coordinates, we bring an equivalent global and fully algebraic description. We show that graded Poisson structures correspond to skew-symmetric (even case) or symmetric (odd case) bivector fields. 

In \S 2, we define quadratic graded vector bundles and the main subject of this paper, graded Courant algebroids. Their basic properties following from the axioms are deduced. We prove that the generalized tangent bundle $T[\ell]\M \oplus T^{\ast}\M$ can be equipped with a canonical structure of a graded Courant algebroid of degree $\ell$, generalizing a notion of a well-known Dorfman bracket (twisted by a closed $3$-form $H$). We show how to generalize the Ševera classification \cite{Severa:2017oew} of exact Courant algebroids to the graded case.

Dirac structures are one of the most important sub-structures of Courant algebroids. In \S 3, we define their graded analogue. This requires one to recall the notion of subbundles of graded vector bundles. We examine maximal isotropic subspaces of quadratic vector spaces and use this notion to define maximal isotropic subbundles. Finally, we define Dirac structures on graded Courant algebroids. We prove that a graph of a vector bundle map $\Pi^{\sharp}: T^{\ast}\M \rightarrow T[\ell]\M$ defines a Dirac structure with respect to the $H$-twisted Dorfman bracket, if and only if it corresponds to the unique $H$-twisted graded Poisson structure of degree $\ell$. 

In \S 4, we show how generalized complex structures can be defined as certain involutive subbundles of complexified Courant algebroids. Equivalently, they correspond to orthogonal endomorphisms satisfying $\J^{2} = - \1_{\E}$ whose eigenspaces form subbundles (this has to be assumed in the graded setting) with a vanishing Nijenhuis tensor. We show that a symplectic form $\omega$ of degree $-\ell$, viewed as a vector bundle map $\omega^{\flat}: T[\ell]\M \rightarrow T^{\ast}\M$, can be used to construct an example of a generalized complex structure. Similarly, so does a complex structure $J: T\M \rightarrow T\M$. 

In the graded setting, a graded Courant algebroid $\E$ can be equipped with a degree $1$ differential $\Delta: \Gamma_{\E}(M) \rightarrow \Gamma_{\E}(M)$ forming a graded derivation of the bracket $[\cdot,\cdot]_{\E}$. It turns out that one has to impose a compatibility with other structures. We obtain a notion of differential graded Courant algebroids, discussed in \S 5. We show how so called ``homological sections'' of $\E$ can be used to produce $\Delta$. We classify all differentials on exact graded Courant algebroids. One can impose the $\Delta$-compatibility of Dirac structures and generalized complex structures. We show that this leads to a generalization of QP-manifolds and differential graded symplectic manifolds. 

Finally, in \S 6, we give a definition of a graded Lie algebroid and show how it induces certain structures on its algebras of forms and multivector fields (we have actually moved the details to the appendix \ref{sec_appendixA}). This allows one to generalize a notion of Lie bialgebroids \cite{mackenzie1994, kosmann1995exact} to the graded setting. It turns out that every graded Lie bialgebroid can be used to produce a graded Courant algebroid, called its double. This is a generalization of the famous result \cite{liu1997manin}. We show how the graded Dorfman bracket can be viewed as a double of a Lie bialgebroid. It is argued how a graded Poisson manifold induces a graded Lie bialgebroid. Finally, graded Lie bialgebroids over a single point lead to the notion of graded Lie bialgebras. We show how this definition can be restated in terms of a Chevalley-Eilenberg cohomology of graded Lie algebras and give a non-trivial example of such structure. 
\section{Elements of graded geometry}
In this section, we shall very briefly recall notions of graded geometry required in this paper. For a detailed exposition and examples, we refer the reader to \cite{vysoky2021global}. 

In general, we use the following somewhat unusual approach to graded algebra. Each graded object, say a \textbf{graded vector space} $V$, is always a sequence $V = \{ V_{j} \}_{j \in \Z}$, where $V_{j}$ is an ordinary vector space for each $j \in \Z$. If there a unique $j \in \Z$, such that $v \in V_{j}$, we write simply $v \in V$ and declare $|v| := j$. In particular, we do not consider inhomogeneous elements. Morphisms of graded objects can be described in a similar manner. For example, a \textbf{graded linear map $\varphi: V \rightarrow W$ of degree $|\varphi|$} is a sequence $\varphi = \{ \varphi_{j} \}_{j \in \Z}$, where $\varphi_{j}: V_{j} \rightarrow W_{j + |\varphi|}$ for every $j \in \Z$. We again use a simplified notation $\varphi(v)$ instead of $\varphi_{|v|}(v)$. Graded vector spaces together with graded linear maps of degree zero form a category $\gVect$. By a \textbf{graded dimension} of $V$, we mean a sequence $\gdim(V) := ( \dim(V_{j}))_{j \in \Z}$. We say that $V$ is finite-dimensional, if $\sum_{j \in \Z} \dim(V_{j}) < \infty$. 

For our purposes, the most important category is the one of \textbf{graded commutative associative algebras}, denoted as $\gcAs$, where each object is a graded vector space $A = \{ A_{j} \}_{j \in \Z}$ equipped with an $\R$-bilinear product $\cdot: A \times A \rightarrow A$ having the following properties:
\begin{enumerate}[(i)]
\item It is of degree zero, that is $|a \cdot b| = |a| + |b|$.
\item It is associative, that is $a \cdot (b \cdot c) = (a \cdot b) \cdot c$.
\item There exists a (unique) unit element $1 \in A$ satisfying $a \cdot 1 = 1 \cdot a = a$. One has $|1| = 0$. 
\item It is graded commutative, that is $a \cdot b = (-1)^{|a||b|} b \cdot a$. 
\end{enumerate}
Axioms are assumed to hold for all elements of $A$. Morphisms in the category $\gcAs$ are degree zero graded linear maps preserving products and unit elements. For other examples  appearing throughout this paper, we refer the reader to \S 1 of \cite{vysoky2021global}.

A \textbf{graded manifold} $\M = (M, \C^{\infty}_{\M})$ of a graded dimension $(n_{j})_{j \in \Z}$ is a second countable Hausdorff topological space $M$ together with a sheaf $\C^{\infty}_{\M}$ of graded commutative associative algebras over $M$. Moreover, it has to have the following properties:
\begin{enumerate}[1)]
\item $(M,\C^{\infty}_{\M})$ is a graded locally ringed space, that is each stalk $\C^{\infty}_{\M,m}$ is a local graded ring.
\item There exists a \textbf{graded smooth atlas} $\A = \{ (U_{\alpha}, \varphi_{\alpha}) \}_{\alpha \in I}$, that is $M = \cup_{\alpha \in I} U_{\alpha}$, and for each $\alpha \in I$, $\varphi_{\alpha}: \M|_{U_{\alpha}} \rightarrow \hat{U}_{\alpha}^{(n_{j})}$ is an isomorphism of graded locally ringed spaces. By $\hat{U}_{\alpha}^{(n_{j})}$ we denote a graded domain over $\hat{U}_{\alpha} \subseteq \R^{n_{0}}$ of a graded dimension $(n_{j})_{j \in \Z}$, see the example below. Every such pair $(U,\varphi)$ is called a \textbf{graded chart} for $\M$ over $U$.
\end{enumerate}
Sections of the sheaf $\C^{\infty}_{\M}$ are usually called \textbf{functions on a graded manifold $\M$}. 
\begin{example}
Let $(n_{j})_{j \in \Z}$ be a sequence of non-negative integers, such that $n := \sum_{j \in \Z} n_{j} < \infty$. Let $n_{\ast} = n - n_{0}$ and consider a collection of variables $\{ \xi_{\mu} \}_{\mu = 1}^{n_{\ast}}$ where each of them is assigned a degree $|\xi_{\mu}| \in \Z$. Suppose that $n_{j} = \#\{ \mu \in \{1, \dots, n_{\ast} \} \; | \; |\xi_{\mu}| = j \}$ for each $j \in \Z - \{0\}$. The variables are assumed to commute according to the rule
\begin{equation} \label{eq_commrelation}
\xi_{\mu} \xi_{\nu} = (-1)^{|\xi_{\mu}||\xi_{\nu}|} \xi_{\nu} \xi_{\mu}.
\end{equation}
Let $\hat{U} \in \Op(\R^{n_{0}})$ be an open subset of $\R^{n_{0}}$. Elements of $\C^{\infty}_{(n_{j})}(\hat{U}) \in \gcAs$ are formal power series in $\{ \xi_{\mu} \}_{\mu=1}^{n_{\ast}}$ with coefficients in the algebra $\C^{\infty}_{n_{0}}(\hat{U})$ of smooth functions on $\hat{U}$. In other words, one has $f \in \C^{\infty}_{(n_{j})}(\hat{U})$ of degree $|f|$, if 
\begin{equation} \label{eq_formalpowseries}
f = \sum_{\fp \in \N_{|f|}^{n_{\ast}}} f_{\fp} \xi^{\fp},
\end{equation}
where $N_{|f|}^{n_{\ast}} = \{ \fp = (p_{1}, \dots, p_{n_{\ast}}) \in (\N_{0})^{n_{\ast}} \; | \; \sum_{\mu=1}^{n_{\ast}} p_{\mu} |\xi_{\mu}| = |f| \text{ and } p_{\mu} \in \{0,1\} \text{ for } |\xi_{\mu}| \text{ odd} \}$, $f_{\fp} \in \C^{\infty}_{n_{0}}(\hat{U})$ and $\xi^{\fp} = (\xi_{1})^{p_{1}} \dots (\xi_{n_{\ast}})^{p_{n_{\ast}}}$. We declare $(\xi_{\mu})^{0} = 1$. This is a unique representation of every formal power series of degree $|f|$ due to (\ref{eq_commrelation}) and the fact that $\xi_{\fp} \neq 0$ for each $\fp \in \N^{n_{\ast}}_{|f|}$. 

Let $f,g \in \C^{\infty}_{(n_{j})}(\hat{U})$. Their product $f \cdot g$ is an element of degree $|f| + |g|$ given by the formula
\begin{equation}
f \cdot g := \hspace{-2mm} \sum_{\fp \in \N^{n_{\ast}}_{|f|+|g|}} \hspace{-2mm} (f \cdot g)_{\fp} \xi^{\fp}, \text{ where } (f \cdot g)_{\fp} := \sum_{\substack{\fq \in \N^{n_{\ast}}_{|f|} \\ \fq \leq \fp}} \epsilon^{\fq,\fp-\fq} f_{\fq} g_{\fp-\fq},
\end{equation}
and $\epsilon^{\fq,\fp-\fq} \in \{-1,1\}$ is the unique sign obtained by reordering $\xi^{\fq} \xi^{\fp-\fq}$ into $\xi^{\fp}$ in accordance with (\ref{eq_commrelation}). In other words, up to signs, $f \cdot g$ is a usual product of formal power series. It is not difficult to see that this makes $\C^{\infty}_{(n_{j})}(\hat{U})$ into a graded commutative associative algebra (with the unit). This algebra can be defined in a more abstract way, see \S 1.3 of \cite{vysoky2021global}.

Now, let $\hat{U} \in \Op(\R^{n_{0}})$ be fixed. For each $\hat{V} \in \Op(\hat{U})$, the assignment 
\begin{equation}
\hat{V} \mapsto \C^{\infty}_{(n_{j})}(\hat{V})
\end{equation}
 makes $\C^{\infty}_{(n_{j})}$ into a sheaf on $\hat{U}$ valued in $\gcAs$. Sheaf restrictions are given by restricting the coefficient functions in the expansion (\ref{eq_formalpowseries}). In fact, $\hat{U}^{(n_{j})} := (\hat{U}, \C^{\infty}_{(n_{j})})$ becomes a graded locally ringed space called a \textbf{graded domain of a graded dimension $(n_{j})_{j \in \Z}$}. 
\end{example}
Let $\M = (M,\C^{\infty}_{\M})$ and $\cN = (N, \C^{\infty}_{\cN})$ be a pair of graded manifolds. We say that $\phi: \M \rightarrow \cN$ is a \textbf{graded smooth map}, if $\phi = (\ul{\phi}, \phi^{\ast})$ for a continuous map $\ul{\phi}: M \rightarrow N$ and a sheaf morphism $\phi^{\ast}: \C^{\infty}_{\cN} \rightarrow \ul{\phi}_{\ast} \C^{\infty}_{\M}$. In other words, for each $V \in \Op(N)$, we obtain a graded algebra morphism $\phi^{\ast}_{V}: \C^{\infty}_{\cN}(V) \rightarrow \C^{\infty}_{\M}( \ul{\phi}^{-1}(V))$ natural in $V$. Finally, the map $[g]_{\ul{\phi}(m)} \mapsto [\varphi^{\ast}_{V}(g)]_{m}$ must define a morphism of local graded rings for all $m \in M$, $V \in \Op_{\ul{\phi}(m)}(N)$ and $g \in \C^{\infty}_{\cN}(V)$. 

Hence in the above definition, $\varphi_{\alpha}: \M|_{U_{\alpha}} \rightarrow \hat{U}^{(n_{j})}_{\alpha}$ becomes a graded smooth map. It is easy to see that $\ul{\A} = \{ (U_{\alpha}, \ul{\varphi_{\alpha}}) \}_{\alpha \in I}$ defines an ordinary smooth atlas on $M$, hence making it into an $n_{0}$-dimensional smooth manifold. It then follows that for any graded smooth map $\phi: \M \rightarrow \cN$, the underlying continuous map $\ul{\phi}: M \rightarrow N$ is smooth. 

One can view $M$ as a (trivially) graded manifold. There exists a canonical graded smooth map $i: M \rightarrow \M$ over $\1_{M}: M \rightarrow M$. For each $U \in \Op(M)$ and $f \in \C^{\infty}_{\M}(M)$, the ordinary smooth function $\ul{f} := i^{\ast}_{U}(f) \in \C^{\infty}_{M}(U)$ is called the \textbf{body of $f$}. For any $m \in U$, we then write $f(m) := \ul{f}(m)$. Note that $\ul{f} = 0$ whenever $|f| \neq 0$.  

For $\M = \hat{U}^{(n_{j})}$, $\hat{V} \in \Op(\hat{U})$ and $f \in \C^{\infty}_{(n_{j})}(\hat{V})$ in the form (\ref{eq_formalpowseries}) with $|f| = 0$, one has $\ul{f} := f_{\mathbf{0}}$, where $\mathbf{0} := (0,\dots,0)$. For general $\M$, $\ul{f}$ is defined precisely to have this form when composed with graded charts and one only has to verify that this makes sense globally. 

Vector fields can be defined as sections of the sheaf of graded derivations of $\C^{\infty}_{\M}$. In other words, for each $U \in \Op(M)$ define
\begin{equation}
\X_{\M}(U) := \gDer( \C^{\infty}_{\M}(U)).
\end{equation}
We say that $X \in \gDer(\C^{\infty}_{\M}(U))$ is a \textbf{vector field of degree $|X|$}. By definition, it is a graded linear map $X: \C^{\infty}_{\M}(U) \rightarrow \C^{\infty}_{\M}(U)$ of degree $|X|$ and for all $f,g \in \C^{\infty}_{\M}(U)$, one has 
\begin{equation}
X(fg) = X(f)g + (-1)^{|X||f|} f X(g). 
\end{equation}
$\X_{\M}(U)$ has a natural structure of a graded $\C^{\infty}_{\M}(U)$-module given by $(f \tr X)(g) := f X(g)$. There is a unique way how to make the assignment $U \mapsto \X_{\M}(U)$ into a sheaf of graded $\C^{\infty}_{\M}$-modules, such that for all $V \subseteq U \subseteq M$, $X \in \X_{\M}(U)$ and $f \in \C^{\infty}_{\M}(U)$, one has
\begin{equation}
X|_{V}(f|_{V}) = X(f)|_{V}. 
\end{equation}
Similarly to ordinary differential geometry, one utilizes the partitions of unity to define $X|_{V}$. Moreover, recall that $\X_{\M}(M)$ has a natural structure of a graded Lie algebra (of degree zero) provided by the graded commutator, defined as 
\begin{equation} \label{eq_gcommutator}
[X,Y](f) := X(Y(f)) - (-1)^{|X||Y|} Y(X(f)),
\end{equation}
for all $X,Y \in \X_{\M}(M)$ and $f \in \C^{\infty}_{\M}(M)$. 
\begin{definice}[\textbf{Graded vector bundles over graded manifolds}]
Let $\M = (M,\C^{\infty}_{\M})$ be a graded manifold. Suppose we are given a sheaf of graded $\C^{\infty}_{\M}$-modules $\Gamma_{\E}$ that is locally freely and finitely generated of a finite graded rank $(r_{j})_{j \in \Z}$. In other words, $\Gamma_{\E}$ is a sheaf of graded vector spaces having the following properties:
\begin{enumerate}[(i)]
\item For each $U \in \Op(M)$, $\Gamma_{\E}(U)$ is a graded $\C^{\infty}_{\M}(U)$-module, and restrictions are compatible with the action of $\C^{\infty}_{\M}$.
\item $r := \sum_{j \in \Z} r_{j} < \infty$, and for any $m \in M$, there exists $U \in \Op_{m}(M)$ and a collection $\{ \Phi_{\lambda} \}_{\lambda = 1}^{r}$ of sections in $\Gamma_{\E}(U)$, such that $r_{j} = \# \{ \lambda \in \{1,\dots,r\} \; | \; |\Phi_{\lambda}| = j \}$ for each $j \in \Z$. Moreover, for each $V \in \Op(U)$, every $\psi \in \Gamma_{\E}(V)$ can be uniquely decomposed as $\psi = \sum_{\lambda=1}^{r} f^{\lambda} \Phi_{\lambda}|_{V}$ for $f^{\lambda} \in \C^{\infty}_{\M}(V)$ satisfying $|f^{\lambda}| + |\Phi_{\lambda}| = |\psi|$. 
\end{enumerate}
We say that $\Gamma_{\E}$ is a sheaf of sections of a \textbf{graded vector bundle} $\E$ over $\M$. Its \textbf{graded rank} is defined as $\grk(\E) := (r_{j})_{j \in \Z}$. $\{ \Phi_{\lambda} \}_{\lambda = 1}^{r}$ is called a \textbf{local frame for $\E$ over $U$}. 
\end{definice}
\begin{rem} \label{rem_noSerreswan}
Since we define graded vector bundles rather algebraically, one may be tempted to define graded vector bundles as finitely generated projective graded $\C^{\infty}_{\M}(M)$-modules. However, although $\Gamma_{\E}(M)$ is always a finitely generated projective graded $\C^{\infty}_{\M}(M)$-module, the converse is not true in the graded (or super) setting.
\end{rem}
\begin{example}
Let $K \in \gVect$ be a finite-dimensional graded vector space. For each $U \in \Op(M)$, let $\Gamma_{\M \times K}(U) := \C^{\infty}_{\M}(U) \otimes_{\R} K$. Suppose that $(r_{j})_{j \in \Z} := \gdim(K)$ and let $r := \sum_{j \in \Z} r_{j}$. Let $(\vartheta_{\lambda})_{\lambda=1}^{r}$ be a total basis for $K$. By definition $\# \{ \lambda \in \{1,\dots,r\} \; | \; |\vartheta_{\lambda}| = j \} = r_{j}$ for each $j \in \Z$. It follows that $\{ 1 \otimes \vartheta_{\lambda} \}_{\lambda =1}^{r}$ forms a local frame for $\M \times K$ over $M$. Hence $\M \times K$ forms a graded vector bundle, called a \textbf{trivial vector bundle}.

Note that every local frame for a graded vector bundle $\E$ over $U$ is equivalent to the $\C^{\infty}_{\M}|_{U}$-linear sheaf isomorphism $\Gamma_{\E}|_{U} \cong \Gamma_{\M \times \R^{(r_{j})}}|_{U}$. Recall that $(\R^{(r_{j})})_{k} := \R^{r_{k}}$ for all $k \in \Z$. 
\end{example}
Let $(n_{j})_{j \in \Z}$ be a graded dimension of $\M$ and suppose $\E$ is a graded vector bundle over $\M$ of a graded rank $(r_{j})_{j \in \Z}$. One can show that there is a unique (up to a graded diffeomorphism) graded manifold $\E = (E, \C^{\infty}_{\E})$ of a graded dimension $(n_{j} + r_{-j})_{j \in \Z}$ together with a graded smooth map $\pi: \E \rightarrow \M$. $\E$ is called the \textbf{total space} of a graded vector bundle. One can recover the sheaf $\Gamma_{\E}$ from its sheaf of functions $\C^{\infty}_{\E}$. However, we will rarely actually use this graded manifold explicitly. Note that $\ul{\pi}: E \rightarrow M$ is an ordinary vector bundle of rank $r_{0}$. 

\begin{example} \label{ex_tangent}
Let $\Gamma_{T\M} := \X_{\M}$ and suppose $(n_{j})_{j \in \Z} = \gdim(\M)$. 

Let $(U,\varphi)$ be a graded local chart for $\M$ over $U$, that is $\varphi: \M|_{U} \rightarrow \hat{U}^{(n_{j})}$ is a graded diffeomorphism. If $n = \sum_{j \in \Z} n_{j}$ and $n_{\ast} = n - n_{0}$, we have coordinate functions $\{ x^{i} \}_{i=1}^{n_{0}} \subseteq \C^{\infty}_{(n_{j})}(\hat{U})$ of degree $0$ corresponding to standard coordinates in $\R^{n_{0}}$, and ``purely graded'' coordinate functions $\{ \xi_{\mu} \}_{\mu=1}^{n_{\ast}} \subseteq \C^{\infty}_{(n_{j})}(\hat{U})$. It is sometimes convenient to denote them altogether as $\{ \bbz^{A} \}_{A=1}^{n}$. It is customary to use the same symbols to denote the corresponding functions $\bbz^{A} := \varphi^{\ast}_{\hat{U}}(\bbz^{A})$ in $\C^{\infty}_{\M}(U)$. 

It follows that for each $A \in \{1,\dots, n\}$, there is a vector field $\frac{\partial}{\partial \bbz^{A}} \in \X_{\M}(U)$ of degree $-|\bbz^{A}|$ determined uniquely by the formula
\begin{equation}
\frac{\partial}{\partial \bbz^{A}}(\bbz^{B}) := \delta_{A}^{B},
\end{equation}
for every $B \in \{1,\dots,n\}$. One can show that $\{ \frac{\partial}{\partial \bbz^{A}} \}_{A=1}^{n}$ forms a local frame for $T\M$ over $U$, thus making $T\M$ into a graded vector bundle of a graded rank $(n_{-j})_{j \in \Z}$. See \S 4.2 in \cite{vysoky2021global} for details. $T\M$ is called the \textbf{tangent bundle} of $\M$.
\end{example} 
\begin{example} \label{ex_dual}
Let $\E$ be a graded vector bundle over $\M$ of a graded rank $(r_{j})_{j \in \Z}$. Let $\Gamma_{\E}$ be its sheaf of sections. For each $U \in \Op(M)$, let $\Gamma_{\E^{\ast}}(U)$ be a graded vector space of $\C^{\infty}_{\M}(U)$-linear maps from $\Gamma_{\E}(U)$ to $\C^{\infty}_{\M}(U)$, that is $\alpha \in \Gamma_{\E^{\ast}}(U)$, if $\alpha: \Gamma_{\E}(U) \rightarrow \C^{\infty}_{\M}(U)$ is $\R$-linear of degree $|\alpha|$, and 
\begin{equation}
\alpha(f \psi) = (-1)^{|\alpha||f|} f \alpha(\psi),
\end{equation} 
for all $\psi \in \Gamma_{\E}(U)$ and $f \in \C^{\infty}_{\M}(U)$. There is an obvious graded $\C^{\infty}_{\M}(U)$-module structure on $\Gamma_{\E^{\ast}}(U)$ and one can use partitions of unity to make $\Gamma_{\E^{\ast}}$ into a sheaf of $\C^{\infty}_{\M}$-modules. If $\{ \Phi_{\lambda}\}_{\lambda=1}^{r}$ is a local frame for $\E$ over $U$, one can consider $\Phi^{\lambda} \in \Gamma_{\E^{\ast}}(U)$ determined uniquely by the requirement $\Phi^{\lambda}(\Phi_{\kappa}) := \delta^{\lambda}_{\kappa}$ for each $\kappa \in \{1,\dots,r\}$. It follows that $\{ \Phi^{\lambda} \}_{\lambda =1}^{r}$ is a local frame for $\E^{\ast}$ over $U$, making $\Gamma_{\E^{\ast}}$ into a sheaf of sections of the \textbf{graded vector bundle $\E^{\ast}$ dual to $\E$}. Note that $\grk(\E^{\ast}) = (r_{-j})_{j \in \Z}$. 
\end{example}

\begin{example} \label{ex_degreeshift}
Let $\E$ be a graded vector bundle over $\M$ of a graded rank $(r_{j})_{j \in \Z}$. Let $\Gamma_{\E}$ be its sheaf of sections. Let $\ell \in \Z$. For each $U \in \Op(M)$, consider a graded vector space $\Gamma_{\E[\ell]}(U) := \Gamma_{\E}(U)[\ell]$. For each $j \in \Z$, we thus have $(\Gamma_{\E}(U)[\ell])_{j} := (\Gamma_{\E}(U))_{j + \ell}$. 

A graded $\C^{\infty}_{\M}(U)$-module structure on $\Gamma_{\E[\ell]}(U)$ is for each $\psi \in \Gamma_{\E[\ell]}(U)$ and $f \in \C^{\infty}_{\M}(U)$ given by $f \tr' \psi := (-1)^{|f|\ell} f \psi$. On the right-hand side, there is the original action on $\Gamma_{\E}(U)$. It is obvious that every local frame $(\Phi_{\lambda})_{\lambda=1}^{r}$ for $\E$ over $U$ can be viewed as a local frame for $\E[\ell]$ over $U$, whence $\Gamma_{\E[\ell]}$ becomes a sheaf of sections of a \textbf{degree $\ell$ shift $\E[\ell]$} of $\E$. Note that $\grk(\E[\ell]) = (r_{j+\ell})_{j \in \Z}$.
\end{example}

\begin{definice}
Let $\E$ and $\E'$ be two graded vector bundles over $\M$. By a \textbf{vector bundle map $F: \E \rightarrow \E'$ of degree $|F|$}, we mean a $\C^{\infty}_{\M}(M)$-linear map $F: \Gamma_{\E}(M) \rightarrow \Gamma_{\E'}(M)$ of degree $|F|$. In other words, it satisfies
\begin{enumerate}
\item $|F(\psi)| = |F| + |\psi|$ for all $\psi \in \Gamma_{\E}(M)$;
\item for each $\psi \in \Gamma_{\E}(M)$ and $f \in \C^{\infty}_{\M}(M)$, one has $F(f \psi) = (-1)^{|f||F|} f F(\psi)$. 
\end{enumerate}
If the degree is not specified, we assume $|F| = 0$. 
\end{definice}
\begin{rem}
A degree $|F|$ vector bundle map $F: \E \rightarrow \E'$ can be equivalently viewed as a degree $0$ vector bundle map $F: \E \rightarrow \E'[|F|]$ or $F: \E[-|F|] \rightarrow \E'$. 
\end{rem}

\begin{rem} \label{rem_VBmorphisms}
In this paper, we consider only vector bundle maps over the identity $\1_{\M}: \M \rightarrow \M$ for two graded vector bundles over the same graded manifold. One can show that $F$ can be uniquely extended to a $\C^{\infty}_{\M}$-linear sheaf morphism $F: \Gamma_{\E} \rightarrow \Gamma_{\E'}$. Note that there is a definition of vector bundle maps for any two graded vector bundles, formulated using their respective duals, see \S 5.1 of \cite{vysoky2021global}. 
\end{rem}

In graded geometry, there is a convenient way to introduce the exterior algebra sheaf $\Omega_{\M}$ as a sheaf of functions on the graded manifold $T[1+s]\M$. Here $s$ is a large enough non-negative even integer, and $T[1+s]\M$ is the total space corresponding to the degree shift of the tangent bundle $T\M$, see Example \ref{ex_tangent} and Example \ref{ex_degreeshift}. For an explanation of the strange additional degree shift $s$, see \S 6 of \cite{vysoky2021global}. 

However, in this paper we will use an equivalent more algebraic definition. We define $\Omega^{p}_{\M} := \Omega^{p}_{[T\M]}$, using the general framework discussed in detail in Appendix \ref{sec_appendixA} for a degree zero Lie algebroid $(T\M, \1_{T\M}, [\cdot,\cdot])$. See Example \ref{ex_standardgLA}. 

The sheaves $\X^{p}_{\M}$ and $\~\X^{p}_{\M}$ of skew-symmetric and symmetric $p$-vector fields are defined similarly. We will use the symbol $[\cdot,\cdot]_{S}$ to denote the Schouten-Nijenhuis brackets induced on the multivector fields (of both kinds).

\begin{example}[\textbf{Graded Poisson manifolds}] \label{ex_gPoisson} Let $\M$ be a graded manifold. By a \textbf{Poisson bracket of degree $\ell$} on $\M$, we mean an $\R$-bilinear bracket $\{ \cdot,\cdot\}: \C^{\infty}_{\M}(M) \times \C^{\infty}_{\M}(M) \rightarrow \C^{\infty}_{\M}(M)$ having the following properties:
\begin{enumerate}[(i)]
\item $|\{f,g\}| = |f| + |g| + \ell$ for all $f,g \in \C^{\infty}_{\M}(M)$;
\item \label{eq_PB_gsymmetry} $\{f,g\} = - (-1)^{(|f|+\ell)(|g|+\ell)} \{g,f\}$; 
\item \label{eq_PB_Leibniz} $\{f, gh \} = \{f,g\}h + (-1)^{(|f|+\ell)|g|} g \{f,h\}$; 
\item \label{eq_PB-gJacobi} $\{f, \{g,h\}\} = \{\{f,g\}, h\} + (-1)^{(|f|+\ell)(|g|+\ell)} \{g, \{f,h\}\}$. 
\end{enumerate}
$(\M,\{\cdot,\cdot\})$ is called a \textbf{graded Poisson manifold of degree $\ell$}. This is an example of a graded Poisson algebra \cite{cattaneo2018graded} of degree $\ell$.

 As in ordinary Poisson geometry, there should be a skew-symmetric bivector $\Pi \in \X^{2}_{\M}(M)$ of degree $\ell$, such that $\{\cdot,\cdot\}$ is obtained as a \textbf{derived bracket} from $\Pi$ using $[\cdot,\cdot]_{S}$, that is 
\begin{equation} \label{eq_PBasderived}
\{f,g\} = [[f,\Pi]_{S},g]_{S}. 
\end{equation}
However, using the properties (\ref{eq_LASNB_gsymmetry}, \ref{eq_LASNB_gJacobi}) one finds
\begin{equation}
[[f,\Pi]_{S},g]_{S} = [f, [\Pi,g]_{S}]_{S} = - (-1)^{(|f|+\ell)(|g|+\ell) + \ell} [[g,\Pi]_{S},f]_{S} 
\end{equation}
We see that for odd $\ell$, the formula (\ref{eq_PBasderived}) does not define a graded Poisson bracket satisfying (\ref{eq_PB_gsymmetry}) of the above definition. This cannot be fixed by different sign choices in (\ref{eq_PBasderived}). Hence $\{\cdot,\cdot\}$ corresponds to a skew-symmetric bivector $\Pi \in \X^{2}_{\M}(M)$ \textit{only for even $\ell$}. 

It turns out that odd $\ell$, we have to assume that $\Pi \in \~\X^{2}_{\M}(M)$. Indeed, one uses (\ref{eq_LASNB2_gsymmetry}, \ref{eq_LASNB2_gJacobi}) to find
\begin{equation}
[[f,\Pi]_{S},g]_{S} = [f, [\Pi,g]_{S}]_{S} = (-1)^{(|f|+\ell)(|g|+\ell) + \ell} [[g,\Pi]_{S},f]_{S}.
\end{equation} 
This gives a correct sign in (\ref{eq_PB_gsymmetry}) when $\ell$ is odd. 

In both cases, (\ref{eq_PB_Leibniz}) follows from the Leibniz rule for $[\cdot,\cdot]_{S}$, whereas (\ref{eq_PB-gJacobi}) is equivalent to the condition $[\Pi,\Pi]_{S} = 0$. Note that if $\Pi \in \X^{2}_{\M}(M)$ has an odd degree, the condition $[\Pi,\Pi]_{S} = 0$ is trivial. The same happens for $\Pi \in \~\X^{2}_{\M}(M)$ of even degree. 
\end{example}
\begin{rem}
In fact, the bracket $[\cdot,\cdot]_{S}$ for for skew-symmetric multivectors can be viewed as a graded Poisson bracket on $T^{\ast}[1+s']\M$ of degree $-(1 + s')$. For symmetric multivectors, it defines a graded Poisson bracket on $T^{\ast}[s']\M$ of degree $-s'$. 
\end{rem}
\begin{rem}
For all $p \in \N_{0}$, $\Omega^{p}_{\M}$, $\X^{p}_{\M}$ and $\~\X^{p}_{\M}$ form locally freely and finitely generated sheaves of graded $\C^{\infty}_{\M}$-modules, hence they can be viewed as sheaves of sections of graded vector bundles. 
\end{rem}
\section{Graded Courant algebroids}
In this section, we shall generalize the concept of a Courant algebroid \cite{liu1997manin, 1999math.....10078R, Severa:2017oew}. We will then provide a key example generalizing the well-known Dorfman bracket \cite{dorfman1987dirac} of ordinary generalized geometry. 
\begin{definice}
Let $\E$ be a graded vector bundle. By a \textbf{graded symmetric bilinear form on $\E$ of degree $\ell$}, we mean a vector bundle map $g_{\E}: \E \rightarrow \E^{\ast}$ of degree $\ell$, such that $\<\cdot,\cdot\>_{E}$, defined for all $\psi,\psi' \in \Gamma_{\E}(M)$ by the formula
\begin{equation} \label{eq_gEformErelation}
\<\psi,\psi'\>_{\E} := (-1)^{(|\psi|+\ell)\ell} [g_{\E}(\psi)](\psi'),
\end{equation}
is graded symmetric of degree $\ell$ in its inputs, that is
\begin{equation} \label{eq_gradedsymmetricform}
\< \psi,\psi'\>_{\E} = (-1)^{(|\psi|+\ell)(|\psi'|+\ell)} \<\psi',\psi\>_{\E},
\end{equation}
for all $\psi,\psi' \in \Gamma_{\E}(M)$. We say that $g_{\E}$ is a \textbf{fiber-wise metric on $\E$ of degree $\ell$}, if it is a graded vector bundle isomorphism. In this case, we say that $(\E,g_{\E})$ is a \textbf{quadratic graded vector bundle of degree $\ell$}.
\end{definice}
Let $\psi \in \Gamma_{\E}(M)$. Since $g_{\E}(\psi) \in \Gamma_{\E^{\ast}}(M)$ is $\C^{\infty}_{\M}(M)$-linear of degree $|\psi| + \ell$ , we find that for all $\psi' \in \Gamma_{\E}(M)$ and $f \in \C^{\infty}_{\M}(M)$, one has the rule
\begin{equation} \label{eq_metriclinearsecond}
\<\psi, f \psi'\>_{\E} = (-1)^{(|\psi|+\ell)|f|} f \<\psi,\psi'\>_{\E}. 
\end{equation}
On the other hand, as $g_{\E}: \Gamma_{\E}(M) \rightarrow \Gamma_{\E^{\ast}}(M)$ is $\C^{\infty}_{\M}(M)$-linear of degree $\ell$, one gets
\begin{equation} \label{eq_metriclinearfirst}
\<f \psi, \psi'\>_{\E} = f \<\psi, \psi'\>_{\E}. 
\end{equation}
Note that this rule is consistent with (\ref{eq_gradedsymmetricform}) and (\ref{eq_metriclinearsecond}). This also explains the signs in (\ref{eq_gEformErelation}). 
\begin{rem} \label{rem_nondegen}
Since $g_{\E}$ can be equivalently viewed as a vector bundle map $g_{\E}: \E \rightarrow \E^{\ast}[\ell]$ of degree zero, it can be an isomorphism, only if $\grk(\E) = \grk(\E^{\ast}[\ell])$. It thus follows from Example \ref{ex_dual} and Example \ref{ex_degreeshift} that a fiber-wise metric on $\E$ of degree $\ell$ exists, only if $\grk(\E) = (r_{j})_{j \in \Z}$ satisfies the condition $r_{j} = r_{-(j+\ell)}$ for every $j \in \Z$. 
\end{rem}
\begin{example} \label{ex_canpairing}
Let $\E = T[\ell]\M \oplus T^{\ast}\M$. Each section in $\Gamma_{\E}(M)$ of degree $k \in \Z$ is a pair $(X,\xi)$, where $X \in \X_{\M}(M)$ is a vector field of degree $|X| = k + \ell$ and $\xi \in \Omega_{\M}(M)$ is a $1$-form of degree $|\xi| = k$. In other words, we have $|(X,\xi)| = |X| - \ell = |\xi|$. The dual vector bundle $\E^{\ast}$ can be canonically identified with $T^{\ast}[-\ell]\M \oplus T\M$. One has to be a bit careful when doing this. Indeed, suppose $(\xi,X) \in \Omega^{1}_{\M}[-\ell](M) \oplus \X_{\M}(M)$. This means that $\xi \in \Omega^{1}_{\M}(M)$ is a $1$-form of degree $|\xi| = |(\xi,X)| - \ell$ and $X \in \X_{\M}(M)$ is a vector field of degree $|X| = |(\xi,X)|$. If we want to view $(\xi,X)$ as an element of $\Gamma_{\E^{\ast}}(M)$, we must declare its action on every element $(Y,\eta) \in \Gamma_{\E}(M)$. Set
\begin{equation} \label{eq_dualgentangentident}
[(\xi,X)](Y,\eta) := (-1)^{(|\xi|+\ell)\ell} \xi(Y) + (-1)^{|X||\eta|} \eta(X).
\end{equation}
Note that $|(Y,\eta)| = |Y| - \ell = |\eta|$. This must define a $\C^{\infty}_{\M}(M)$-linear map of degree $|(\xi,X)| = |\xi| + \ell = |X|$. We leave the verification to the reader - be aware that $\C^{\infty}_{\M}(M)$-actions are also modified by degree shifts! Note that the identification of $(\xi,X)$ with an element of $\Gamma_{\E^{\ast}}(M)$ is  $\C^{\infty}_{\M}(M)$-linear of degree zero, hence it defines a graded vector bundle isomorphism. We will henceforth use this identification.

Now, let us define $g_{\E}: \E \rightarrow \E^{\ast}$ as follows. For each $(X,\xi) \in \Gamma_{\E}(M)$, set 
\begin{equation}
g_{\E}(X,\xi) := (\xi,X). 
\end{equation}
First, observe that $|g_{\E}(X,\xi)| = |(\xi,X)| = |X| = |(X,\xi)| + \ell$, and for each $f \in \C^{\infty}_{\M}(M)$, one has 
\begin{equation}
\begin{split}
g_{\E}(f \tr (X,\xi)) = & \ g_{\E}(  (-1)^{|f|\ell}f X, f \xi) = (f \xi, (-1)^{|f|\ell} f X) \\
= & \ (-1)^{|f|\ell} ((-1)^{-|f|\ell} f \xi, fX) = (-1)^{|f|\ell} f \tr (\xi, X) \\
= & \ (-1)^{|f|\ell} f \tr g_{\E}(X,\xi). 
\end{split}
\end{equation}
This proves that $g_{\E}$ is a $\C^{\infty}_{\M}(M)$-linear isomorphism of degree $\ell$. Using (\ref{eq_dualgentangentident}), the formula (\ref{eq_gEformErelation}) gives
\begin{equation} \label{eq_canpairing}
\<(X,\xi), (Y,\eta) \>_{\E} = \xi(Y) + (-1)^{|X||Y|} \eta(X),
\end{equation}
for all $(X,\xi), (Y,\eta) \in \Gamma_{\E}(M)$. It is easy to see that $\<\cdot,\cdot\>_{\E}$ satisfies (\ref{eq_gradedsymmetricform}). We conclude that $g_{\E}$ defines a fiber-wise metric of degree $\ell$ on $\E$. 
\end{example} 
We can now proceed to the main definition of this section. Let $\ell \in \Z$ be any integer. 
\begin{definice} \label{def_gCourant}
$(\E,\rho,g_{\E},[\cdot,\cdot]_{\E})$ is called a \textbf{graded Courant algebroid of degree $\ell$}, if 
\begin{enumerate}[(i)]
\item $\E$ is a graded vector bundle over $\M$;
\item $\rho: \E \rightarrow T\M$ is a vector bundle map of degree $\ell$ called the \textbf{anchor};
\item $g_{\E}$ is a fiber-wise metric on $\E$ of degree $\ell$; 
\item $[\cdot,\cdot]_{\E}: \Gamma_{\E}(M) \times \Gamma_{\E}(M) \rightarrow \Gamma_{\E}(M)$ is an $\R$-bilinear map of degree $\ell$, that is $|[\psi,\psi']_{\E}| = |\psi| + |\psi'| + \ell$ for all $\psi,\psi' \in \Gamma_{\E}(M)$. 
\end{enumerate}
Moreover, the operations are subject to the following set of axioms. All conditions have to hold for all involved sections and functions. Let $\hat{\rho}(\psi) := (-1)^{(|\psi|+\ell)\ell} \rho(\psi)$. 
\begin{enumerate}[(c1)]
\item The fiber-wise metric and the bracket are compatible, that is 
\begin{equation} \label{eq_CApairingcomp}
\hat{\rho}(\psi)\<\psi',\psi''\>_{\E} = \< [\psi,\psi']_{\E}, \psi''\>_{\E} + (-1)^{(|\psi|+\ell)(|\psi'|+\ell)} \< \psi', [\psi,\psi'']_{\E} \>_{\E}.
\end{equation}
\item The bracket is graded skew-symmetric up to an extra term, that is 
\begin{equation} \label{eq_CAskewsym}
[\psi,\psi']_{\E} + (-1)^{(|\psi|+\ell)(|\psi'|+\ell)} [\psi',\psi]_{\E} = (-1)^{|\psi|+|\psi'|} \dD \<\psi,\psi'\>_{\E},
\end{equation}
where $\dD: \C^{\infty}_{\M}(M) \rightarrow \Gamma_{\E}(M)$ is a degree zero $\R$-linear map defined as $\dD := g_{\E}^{-1} \circ \rho^{T} \circ \dr$.
\item The bracket satisfies the \textbf{graded Jacobi identity}
\begin{equation} \label{eq_CAgJI}
[\psi,[\psi',\psi'']_{\E}]_{\E} = [[\psi,\psi']_{\E}, \psi'']_{\E} + (-1)^{(|\psi|+\ell)(|\psi'|+\ell)} [\psi', [\psi, \psi'']_{\E}]_{\E}.
\end{equation}
\end{enumerate}
\end{definice}

Similarly to ordinary Courant algebroids, one can find some immediate consequences of the axioms. Let us name a few.
\begin{tvrz} \label{tvrz_CAconsequences}
The bracket satisfies the \textbf{graded Leibniz rule} 
\begin{equation} \label{eq_CALeibniz}
[\psi, f\psi']_{\E} = \hat{\rho}(\psi)(f) \psi' + (-1)^{|f|(|\psi|+\ell)} f [\psi,\psi']_{E},
\end{equation}
for all $\psi,\psi' \in \Gamma_{\E}(M)$ and $f \in \C^{\infty}_{\M}(M)$. The anchor $\rho$ is a bracket homomorphism, that is 
\begin{equation} \label{eq_rhoishom}
\rho( [\psi,\psi']_{\E}) = [\rho(\psi), \rho(\psi')],
\end{equation}
for all $\psi,\psi' \in \Gamma_{\E}(M)$. 
One has $\rho \circ \dD = 0$ and thus $\rho \circ \rho^{\ast} = 0$, where $\rho^{\ast} := g_{\E}^{-1} \circ \rho^{T}$. Consequently, there is a sequence of graded vector bundles over $\M$ and degree zero vector bundle maps over $\1_{\M}$ 
\begin{equation}
\begin{tikzcd} \label{eq_gCAsequence}
0 \arrow{r} & T^{\ast}\M \arrow{r}{\rho^{\ast}} & \arrow{r}{\rho} \E & T[\ell]\M \arrow{r} & 0 
\end{tikzcd}
\end{equation}
forming a cochain complex. Finally, for any $f \in \C^{\infty}_{\M}(M)$ and $\psi \in \Gamma_{\E}(M)$, one has 
\begin{equation} \label{eq_bracketswithD}
[\psi, \dD{f}]_{\E} = (-1)^{|f|+\ell} \dD \<\psi, \dD{f}\>_{\E}, \; \; [\dD{f}, \psi]_{\E} = 0. 
\end{equation}
\end{tvrz}
\begin{proof}
All proofs are completely analogous to the ones for ordinary Courant algebroids, see e.g. \cite{1999math.....10078R, Kotov:2010wr, kosmann2005quasi}. The equation (\ref{eq_CALeibniz}) is implied by (\ref{eq_CApairingcomp}) being consistent with the $\C^{\infty}_{\M}$-bilinearity of $\<\cdot,\cdot\>_{\E}$. The equation (\ref{eq_rhoishom}) follows easily from (\ref{eq_CAgJI}) together with the graded Leibniz rule (\ref{eq_CALeibniz}). Using (\ref{eq_CALeibniz}) and (\ref{eq_CAskewsym}), one can derive the rule
\begin{equation}
[f\psi,\psi']_{\E} = f[\psi,\psi']_{\E} - (-1)^{(|f|+|\psi|+\ell)(|\psi'|+\ell)} \hat{\rho}(\psi')(f) \psi + (-1)^{\ell + |f|(|\psi|+|\psi'|+\ell + 1)} \<\psi,\psi'\>_{\E} \dD{f}.
\end{equation}
Applying $\rho$ on both sides of this equation, using (\ref{eq_rhoishom}) and the fact that $g_{\E}$ is an isomorphism, one can prove that $\rho \circ \dD = 0$. Since $\dD{f} = \rho^{\ast}( \dr{f})$ and $\dr{f}$ locally generate $\Omega^{1}_{\M}$, this already implies $\rho \circ \rho^{\ast} = 0$. The claim about (\ref{eq_gCAsequence}) follows immediately. To prove the first equation in (\ref{eq_bracketswithD}), one uses the fact that $\<\phi,\dD{f}\>_{\E} = (-1)^{|f|+\ell} \hat{\rho}(\phi)f$ for all $\phi \in \Gamma_{\E}(M)$ and $f \in \C^{\infty}_{\M}(M)$. By applying $\hat{\rho}(\psi)$ on this equation and using (\ref{eq_CApairingcomp}) and (\ref{eq_rhoishom}), one finds $\<\phi, [\psi,\dD{f}]_{\E} \> = (-1)^{|f|+\ell} \<\phi, \dD{ \<\psi, \dD{f}\>_{\E}} \>_{\E}$. The claim then follows from the fact that $g_{\E}$ is an isomorphism. The other of the two equations then follows from the first one together with (\ref{eq_CAskewsym}). 
\end{proof}
\begin{rem}
We have generalized the definition \cite{kosmann2005quasi} with a minimal number of requirements. In other literature, some of the consequences (\ref{eq_CALeibniz}, \ref{eq_rhoishom}) are given as axioms. 
\end{rem}
\begin{rem}
Our main intention was to have a bracket operation $[\cdot,\cdot]_{\E}$ of degree $\ell$. The graded Leibniz rule (\ref{eq_CALeibniz}) then requires the anchor $\rho$ to be a degree $\ell$ map. Axioms (\ref{eq_CApairingcomp}) and (\ref{eq_CAskewsym}) could be in principle formulated for a fiber-wise metric of an arbitrary (independent of $\ell$) degree. However, this freedom can be eliminated by considering a degree shifted vector bundle $\E[k]$ where $k \in \Z$ can be always chosen so that $g_{\E}$ has the same degree as $[\cdot,\cdot]_{\E}$. 

Note that after assuming that the bracket and $g_{\E}$ have the same degree, no further degree shifting is allowed. In other words, the degree $\ell$ is an ``intrinsic'' property of graded Courant algebroids. Observe that $[\cdot,\cdot]_{\E}$ defines a so called Loday bracket on $\Gamma_{\E}(M)$, see \S 2.1 of \cite{Kosmann1996}. 
\end{rem}

It is obvious that the definition of graded Courant algebroid was tailored in order for the following pivotal example to fit in. Let $\ell \in \Z$ be fixed.
\begin{example}[\textbf{Degree $\ell$ Dorfman bracket}] \label{ex_gDorfman}
Let $\E = T[\ell]\M \oplus T^{\ast}\M$ be equipped with the fiber-wise metric $g_{\E}$ introduced in Example \ref{ex_canpairing}. Let $H \in \Omega^{3}_{\M}(M)$ with $|H| = -\ell$. Set
\begin{equation} \label{eq_gDorfman}
[(X,\xi),(Y,\eta)]_{D}^{H} := ([X,Y], (-1)^{|X|\ell} \Li{X}\eta - (-1)^{|Y|(|X|+\ell)} i_{Y} \dr{\xi} + (-1)^{\ell} H(X,Y,\cdot)),
\end{equation}
for all $(X,\xi),(Y,\eta) \in \Gamma_{\E}(M)$. Let $\rho(X,\xi) := X$ for all $(X,\xi) \in \Gamma_{\E}(M)$. We claim that $(\E, \rho, g_{\E}, [\cdot,\cdot]_{D}^{H})$ becomes a graded Courant algebroid of degree $\ell$, if and only if $\dr{H} = 0$. 

First, observe that $|\rho(X,\xi)| = |X| = |(X,\xi)| + \ell$, whence $\rho$ is a map of degree $\ell$. One has
\begin{equation}
\rho(f \tr (X,\xi)) = \rho( (-1)^{|f|\ell} fX, f\xi) = (-1)^{|f| \ell} fX = (-1)^{|f| \ell} f \rho(X,\xi),
\end{equation}
for all $(X,\xi) \in \Gamma_{\E}(M)$ and $f \in \C^{\infty}_{\M}(M)$. This proves that $\rho$ is $\C^{\infty}_{\M}(M)$-linear of degree $\ell$. 

Next, for every $\alpha \in \Omega^{1}_{\M}(M)$ and every $(X,\xi) \in \Gamma_{\E}(M)$, one finds 
\begin{equation}
[ \rho^{T}(\alpha)](X,\xi) \equiv (-1)^{|\alpha|\ell} \alpha( \rho(X,\xi)) = (-1)^{|\alpha| \ell} \alpha(X).
\end{equation}
Using (\ref{eq_dualgentangentident}), we can thus write $\rho^{T}(\alpha) = ((-1)^{\ell} \alpha, 0) \in (\Omega^{1}_{\M}[-\ell] \oplus \X_{\M})(M)$. Hence
\begin{equation}
\rho^{\ast}(\alpha) = (g_{\E}^{-1} \circ \rho^{T})(\alpha) = g_{\E}^{-1}( (-1)^{\ell} \alpha, 0) = (0, (-1)^{\ell} \alpha). 
\end{equation} 
In particular, we find that $\dD{f} = (0, (-1)^{\ell} \dr{f})$ for every $f \in \C^{\infty}_{\M}(M)$. This is indeed an $\R$-linear map of degree zero. 

Let us turn our attention to the bracket. Recall that for any $(X,\xi) \in \Gamma_{\E}(M)$, we have $|(X,\xi)| = |X| - \ell = |\xi|$. We must verify that $[(X,\xi),(Y,\eta)]_{D}^{H}$ is a section of $\E$ of degree $|(X,\xi)| + |(Y,\eta)| + \ell$. In particular, $[X,Y]$ must be a vector field of degree $(|(X,\xi)| + |(Y,\eta)| + \ell) + \ell$. But this is true as 
\begin{equation}
|[X,Y]| = |X| + |Y| = (|(X,\xi)| + \ell) + (|(Y,\eta)| + \ell).
\end{equation}
Next, all $1$-forms on the right-hand side of (\ref{eq_gDorfman}) must have a degree $|(X,\xi)| + |(Y,\eta)| + \ell$. But 
\begin{align}
|\Li{X}\eta| = & \ |X| + |\eta| = (|(X,\xi)| + \ell)) + |(Y,\eta)|, \\
|i_{Y}\dr{\xi}| = & \ |Y| + |\xi| = (|(Y,\eta)| + \ell) + |(X,\xi)|, \\
|H(X,Y,\cdot)| = & \ |X| + |Y| + |H| = (|(X,\xi)| + \ell) + (|(Y,\eta)| + \ell) - \ell. 
\end{align}
Note that $H(X,Y,\cdot) \in \Omega^{1}_{\M}(M)$ can be also written using the interior product:
\begin{equation}
H(X,Y,\cdot) = (-1)^{-(|X|-1)\ell + (|Y|-1)(|X| - \ell)} i_{Y} i_{X}H. 
\end{equation}
We leave the verification of (\ref{eq_CApairingcomp}) and (\ref{eq_CAskewsym}) to the reader. One employs just the Cartan relations (\ref{eq_LACartan1} - \ref{eq_LACartan5}) and definitions in the process. Finally, after a significant amount of work (especially to get the correct signs), one can show that the graded Jacobi identity (\ref{eq_CAgJI}) holds for sections $(X,\xi),(Y,\eta),(Z,\zeta) \in \Gamma_{\E}(M)$, if and only if the section $(0, \dr{H}(X,Y,Z,\cdot)) \in \Gamma_{\E}(M)$ vanishes. 

We conclude that $(\E,\rho,g_{\E},[\cdot,\cdot]_{D}^{H})$ forms a graded Courant algebroid of degree $\ell$, if and only if $\dr{H} = 0$. We call $[\cdot,\cdot]_{D}^{H}$ a \textbf{degree $\ell$ graded Dorfman bracket twisted by $H$}. 

Let us conclude this example with the following observation. Let $\omega \in \Omega^{2}_{\M}(M)$ with $|\omega| = - \ell$. One can use it to define a $\C^{\infty}_{\M}(M)$-linear vector bundle map $\omega^{\flat}: T[\ell]\M \rightarrow T^{\ast}\M$ of degree zero given by $[\omega^{\flat}(X)](Y) := \omega(X,Y)$. Equivalently, $\omega^{\flat}(X) = (-1)^{(|X|-1)\ell} i_{X}\omega$. Define
\begin{equation} \label{eq_etoomega}
e^{\omega}(X,\xi) := (X, \xi + \omega^{\flat}(X)),
\end{equation}
for all $(X,\xi) \in \Gamma_{\E}(M)$. It follows that $e^{\omega}: \E \rightarrow \E$ is a degree zero vector bundle isomorphism with the inverse $e^{-\omega}$. Using definitions and the Cartan relations (\ref{eq_LACartan1}, \ref{eq_LACartan3}), one can show that 
\begin{equation} \label{eq_gDorfmanomegatwist}
e^{\omega}( [(X,\xi),(Y,\eta)]_{D}^{H + \dr{\omega}}) = [e^{\omega}(X,\xi), e^{\omega}(Y,\eta)]_{D}^{H},
\end{equation}
for all $(X,\xi),(Y,\eta) \in \Gamma_{\E}(M)$. Also note that $\rho \circ e^{\omega} = \rho$ and $\< e^{\omega}(X,\xi), e^{\omega}(Y,\eta)\>_{\E} = \<(X,\xi),(Y,\eta)\>_{\E}$.  Note that (\ref{eq_gDorfmanomegatwist}) was one of the most useful guides towards correct signs in (\ref{eq_gDorfman}) and Definition \ref{def_gCourant}. 
\end{example}
One obtains the following graded version of the classification \cite{Severa:2017oew} of exact Courant algebroids.
\begin{tvrz}[\textbf{Graded Ševera classification}]
We say that a graded Courant algebroid $(\E,\rho,g_{\E},[\cdot,\cdot]_{\E})$ of degree $\ell$ is \textbf{exact}, if the sequence (\ref{eq_gCAsequence}) is exact. 

Then isomorphism classes of exact Courant algebroids of degree $\ell$ over $\M$ are in one-to-one correspondence with de Rham cohomology classes $[H] \in H^{3}_{\M}(M)$ of degree $-\ell$. In particular, for $\ell \neq 0$, there is a unique isomorphism class of exact graded Courant algebroids of degree $\ell$. 
\end{tvrz}
\begin{proof}
First, recall that an isomorphism of two graded Courant algebroids $(\E,\rho,g_{\E},[\cdot,\cdot]_{\E})$ and $(\E', \rho', g_{\E'}, [\cdot,\cdot]_{\E'})$ is a degree zero vector bundle isomorphism $F: \E \rightarrow \E'$, such that 
\begin{equation}
\rho' \circ F = \rho, \; \; \<F(\psi),F(\psi')\>_{\E'} = \<\psi,\psi'\>_{\E}, \; \; F([\psi,\psi']_{\E}) = [F(\psi),F(\psi')]_{\E'},
\end{equation}
for all $\psi,\psi' \in \Gamma_{\E}(M)$. Next, for each $p \in \Z$, the exterior differential $\dr: \Omega^{p}_{\M}(M) \rightarrow \Omega^{p+1}_{\M}(M)$ preserves degrees. It follows that the $p$-th de Rham cohomology $H^{p}_{\M}(M)$ naturally becomes a graded vector space. However, it was observed in \cite{roytenberg2002structure} that its only non-zero component is the zeroth one, that is $[\omega] \neq 0$, only if $|\omega| = 0$. This explains the last sentence of the statement.

Now, let $(\E,\rho,g_{\E},[\cdot,\cdot]_{\E})$ be a given exact graded Courant algebroid of degree $\ell$. One can use partitions of unity to show that every short exact sequence (\ref{eq_gCAsequence}) splits, that is there exists a degree zero vector bundle map $s: T[\ell]\M \rightarrow \E$ such that $\rho \circ s = \1_{T[\ell]\M}$. One can always modify it so that $\<s(X),s(Y)\>_{\E} = 0$ for all $X,Y \in \X_{\M}(M)$. Such $s$ is called an \textit{isotropic splitting}.

To every such splitting, one can assign a degree zero vector bundle isomorphism $\Psi_{s}: T[\ell]\M \oplus T^{\ast}\M \rightarrow \E$ given by $\Psi_{s}(X,\xi) = s(X) + (-1)^{\ell} \rho^{\ast}(\xi)$. One utilizes $\Psi_{s}$ to induce a graded Courant algebroid structure on $T[\ell]\M \oplus T^{\ast}\M$ by declaring it to be an isomorphism of graded Courant algebroids. It is a straightforward calculation to show that in this way, one obtains precisely the structure from Example \ref{ex_gDorfman} for  closed $H_{s} \in \Omega^{3}_{\M}(M)$ of degree $-\ell$ obtained from the formula 
\begin{equation}
[s(X),s(Y)]_{\E} = s([X,Y]) + \rho^{\ast}( H_{s}(X,Y,\cdot)).
\end{equation}
This shows that every exact graded Courant algebroid of degree $\ell$ is isomorphic to the one in Example \ref{ex_gDorfman} for some $H_{s} \in \Omega^{3}_{\M}(M)$ of degree $-\ell$. For any other isotropic splitting $s': T[\ell]\M \rightarrow \E$, there is a unique $\omega \in \Omega^{2}_{\M}(M)$ with $|\omega| = -\ell$, such that $s'$ and $s$ are related by
\begin{equation}
s'(X) = s(X) + (-1)^{\ell} \rho^{\ast}( \omega^{\flat}(X)),
\end{equation} 
for all $X \in \X_{\M}(M)$. Consequently, one has $\Psi_{s'} = \Psi_{s} \circ e^{\omega}$, see (\ref{eq_etoomega}). It follows from (\ref{eq_gDorfmanomegatwist}) that $H_{s'} = H_{s} + \dr{\omega}$. Hence to every exact graded Courant algebroid $\E$ of degree $\ell$, we may assign a unique class in $H^{3}_{\M}(M)$ of degree $-\ell$. If $F: \E \rightarrow \E'$ is an isomorphism of exact graded Courant algebroids, every isotropic splitting $s: T[\ell]\M \rightarrow \E$ induces an isotropic splitting $s' := F \circ s: T[\ell]\M \rightarrow \E'$ such that $H_{s'} = H_{s}$. 

Conversely, to any class $[H] \in H^{3}_{\M}(M)$ of degree $-\ell$, one can use its representative $H \in \Omega^{3}_{\M}(M)$ to construct an exact graded Courant algebroid of degree $\ell$ as in Example \ref{ex_gDorfman}. A different choice of a representative $H' = H + \dr{\omega}$ induces their isomorphism $e^{\omega}$ due to (\ref{eq_gDorfmanomegatwist}). 
\end{proof}
\section{Dirac structures}
Let us start on graded linear algebra level. Suppose $V \in \gVect$ is a finite-dimensional graded vector space together with a degree $\ell$ isomorphism $g_{V}: V \rightarrow V^{\ast}$, such that 
\begin{equation}
\<v,w\>_{V} := (-1)^{(|v|+\ell)\ell} [g_{V}(v)](w)
\end{equation}
induces a graded symmetric bilinear form $\<\cdot,\cdot\>_{V}: V \times V \rightarrow \R$ of degree $\ell$, that is 
\begin{equation}
\<v,w\>_{V} = (-1)^{(|v|+\ell)(|w|+\ell)} \<w,v\>_{V},
\end{equation}
for all $v,w \in V$. Note that $\<v,w\>_{V} \neq 0$, only if $|v|+|w|+\ell = 0$. If $\gdim(V) = (r_{j})_{j \in \Z}$, then $g_{V}$ is an isomorphism, only if $r_{j} = r_{-(j+\ell)}$ for all $j \in \Z$. We say that $(V,g_{V})$ is a \textbf{quadratic graded vector space of degree $\ell$}.

Now, let $L \subseteq V$ be a graded subspace with $\gdim(L) = (q_{j})_{j \in \Z}$. First, we define the \textbf{annihilator of $L$} for each $j \in \Z$ by
\begin{equation}
\an(L)_{j} := \an(L_{-j}) \equiv \{ \alpha \in (V^{\ast})_{j} \; | \; \alpha(v) = 0 \text{ for all } v \in L_{-j} \}.
\end{equation}
This defines a graded subspace $\an(L) \subseteq V^{\ast}$ with $\gdim(\an(L)) = (r_{-j} - q_{-j})_{j \in \Z}$. Let 
\begin{equation}
L^{\perp} := g_{V}^{-1}(\an(L)) = \{ v \in V \; | \; \<v,w\>_{V} = 0 \text{ for all } w \in L \}.
\end{equation}
$L^{\perp}$ is called the \textbf{orthogonal complement to $L$}. One finds $\gdim(L^{\perp}) = ( r_{-(j+\ell)} - q_{-(j+\ell)})_{j \in \Z}$.

\begin{definice}
Let $(V,g_{V})$ be a quadratic graded vector space of degree $\ell$, and let $L \subseteq V$ be a graded subspace. We say that $L$ is \textbf{isotropic}, if $L \subseteq L^{\perp}$. An isotropic graded subspace is \textbf{maximal}, if $L$ is not properly contained in any isotropic graded subspace. We say that $L$ is a \textbf{Lagrangian graded subspace}, if $L = L^{\perp}$. 
\end{definice}

The proof of the following proposition is similar to the ordinary case and can be deduced from definitions and properties of ordinary bilinear forms. We leave it to the reader. 
\begin{tvrz} \label{tvrz_maxisos1}
Let $(V,g_{V})$ be a quadratic graded vector space of degree $\ell$. Let $L \subseteq V$ be a graded subspace. Let $(r_{j})_{j \in \Z} = \gdim(V)$ and $(q_{j})_{j \in \Z} = \gdim(L)$.
\begin{enumerate}[(1)]
\item For $\ell \pmod 4 \neq 0$, the following statements are equivalent: 
\begin{enumerate} 
\item $L$ is maximal isotropic;
\item $L$ is Lagrangian;
\item $L$ is isotropic and $r_{j} = q_{j} + q_{-(j+\ell)}$ for all $j \in \Z$. 
\end{enumerate}
\item Let $\ell \pmod 4 = 0$, that is $\ell = 2k$ for even $k \in \Z$. Then $\<\cdot,\cdot\>_{V}$ restricts to a non-degenerate symmetric bilinear form $\<\cdot,\cdot\>_{V}^{0}$ on $V_{-k}$. Let $(s,t)$ be the signature of $\<\cdot,\cdot\>_{V}^{0}$. Then the following statements are equivalent:
\begin{enumerate}
\item $L$ is maximal isotropic;
\item $L$ is Lagrangian (only for $s = t$);
\item $L$ is isotropic, $q_{-k} = \min \{ s,t \}$, and $r_{-k+i} = q_{-k+i} + q_{-k-i}$ for all $i \in \N$.
\end{enumerate}
\end{enumerate}
\end{tvrz}
\begin{rem}
Observe that in \textit{(2)-(c)}, the condition $L \subseteq L^{\perp}$ implies that $L_{-k} \subseteq V_{-k}$ is isotropic with respect to $\<\cdot,\cdot\>_{0}$, and the condition $q_{-k} = \min\{s,t\}$ ensures that $L_{-k}$ is maximal. Note that $(1)$ includes the case where $\ell = 2k$ for odd $k \in \Z$. In this case $\<\cdot,\cdot\>_{0}$ is a non-degenerate skew-symmetric bilinear form on $V_{-k}$ and $L_{-k} \subseteq V_{-k}$ is maximal isotropic, iff it is Lagrangian. 
\end{rem}
Now, let us turn our attention to graded vector bundles. First, we must recall some basic facts about subbundles and their orthogonal complements. 
\begin{definice}
Let $\E$ be a graded vector bundle over $\M$. Let $(r_{j})_{j \in \Z} = \grk(\E)$. A \textbf{subbundle} $\cL \subseteq \E$ of a graded rank $(q_{j})_{j \in \Z}$ consists of the following data:
\begin{enumerate}[(i)]
\item A subsheaf $\Gamma_{\cL} \subseteq \Gamma_{\E}$ of graded $\C^{\infty}_{\M}$-submodules. In particular, $\Gamma_{\cL}(U) \subseteq \Gamma_{\E}(U)$ is a graded $\C^{\infty}_{\M}(U)$-submodule for each $U \in \Op(M)$. 
\item Let $r := \sum_{j \in \Z} r_{j}$ and $q := \sum_{j \in \Z} q_{j}$. For each $m \in M$, there exists $U \in \Op_{m}(M)$ and a local frame $\{ \Phi_{\lambda} \}_{\lambda=1}^{r}$ for $\E$ over $U$, such that $\{ \Phi_{\lambda} \}_{\lambda=1}^{q}$ becomes a local frame for $\cL$ over $U$. We say that $\{ \Phi_{\lambda} \}_{\lambda=1}^{r}$ is \textbf{adapted to $\cL$}. 
\end{enumerate}
In particular, $\Gamma_{\cL}$ is a sheaf of sections of a graded vector bundle $\cL$ of a graded rank $(q_{j})_{j \in \Z}$. 
\end{definice}
For any subbundle $\cL \subseteq \E$, one can define its \textbf{annihilator} $\an(\cL)$. For each $U \in \Op(M)$, let 
\begin{equation} \label{eq_anihillator}
\Gamma_{\an(\cL)}(U) := \an( \Gamma_{\cL}(U)) \subseteq \Gamma_{\E^{\ast}}(U).
\end{equation}
It is not difficult to see that $\Gamma_{\an(\cL)} \subseteq \Gamma_{\E^{\ast}}$ is a subsheaf of graded $\C^{\infty}_{\M}$-submodules. Suppose $\grk(\E) = (r_{j})_{j \in \Z}$ and $\grk(\cL) = (q_{j})_{j \in \Z}$. Let $\{ \Phi_{\lambda} \}_{\lambda=1}^{r}$ be the local frame for $\E$ adapted to $\cL$. It follows that the dual frame $\{ \Phi^{\lambda} \}_{\lambda=1}^{r}$ for $\E^{\ast}$ is adapted to $\an(\cL)$, proving that $\an(\cL)$ is a subbundle of $\E^{\ast}$ of a graded rank $(r_{-j} - q_{-j})_{j \in \Z}$. 

Suppose $(\E,g_{\E})$ is a quadratic graded vector bundle of degree $\ell$. One can thus define the \textbf{orthogonal complement $\cL^{\perp}$ to $\cL$} by declaring 
\begin{equation}
\Gamma_{\cL^{\perp}}(U) := \Gamma_{\cL}(U)^{\perp} \equiv (g_{\E}|_{U})^{-1}\{ \an(\Gamma_{\cL}(U))\},
\end{equation}
where $g_{\E}|_{U}: \Gamma_{\E}(U) \rightarrow \Gamma_{\E^{\ast}}(U)$ is obtained by the restriction of $g_{\E}$, see Remark \ref{rem_VBmorphisms}. Since $\Gamma_{\cL^{\perp}}$ is the image of $\Gamma_{\an(\cL)}$ under a sheaf isomorphism of degree $-\ell$, it defines a subbundle $\cL^{\perp} \subseteq \E$ of a graded rank $(r_{-(j+\ell)} - q_{-(j+\ell)})_{j \in \Z}$. 

To talk about maximal isotropic subbundles, one has to have some notion of a fiber of a graded vector bundle. Since we deal with graded vector bundles only through their sheaves of sections, this requires some work. For each $m \in M$, the \textbf{fiber $\E_{m}$ of $\E$ over $m$} is the graded vector space
\begin{equation}
\E_{m} := \R \otimes_{\C^{\infty}_{\M,m}} \Gamma_{\E,m},
\end{equation}
where $\Gamma_{\E,m} \in \gVect$ is the stalk of the sheaf $\Gamma_{\E}$ at $m$, and $\C^{\infty}_{\M,m} \in \gcAs$ is the stalk of the sheaf of functions $\C^{\infty}_{\M}$ at $m$. $\Gamma_{\E,m}$ has a natural graded $\C^{\infty}_{\M,m}$-module structure induced from $\Gamma_{\E}$. Indeed, let $[f]_{m} \in \C^{\infty}_{\M,m}$ and $[\psi]_{m} \in \Gamma_{\E,m}$ be two germs represented by $f \in \C^{\infty}_{\M}(M)$ and $\psi \in \Gamma_{\E}(M)$. Then
\begin{equation}
[f]_{m} \tr [\psi]_{m} := [f \psi]_{m}.
\end{equation}
The graded vector space $\R$ has the graded $\C^{\infty}_{\M,m}$-module structure given by $[f]_{m} \tr \lambda := f(m)\lambda$ for every $[f]_{m} \in \C^{\infty}_{\M,m}$ and $\lambda \in \R$. Now, for every $U \in \Op_{m}(M)$ and $\psi \in \Gamma_{\E}(U)$, its \textbf{value $\psi|_{m} \in \E_{m}$ at $m$} is defined as $\psi|_{m} := 1 \otimes [\psi]_{m}$. Let us summarize basic facts about fibers. 
\begin{tvrz}
The fibers of graded vector bundles have the following properties:
\begin{enumerate}[(i)]
\item Let $m \in M$ and let $\{ \Phi_{\lambda} \}_{\lambda=1}^{r}$ be a local frame for $\E$ over $U \in \Op_{m}(M)$. Then $\{ \Phi_{\lambda}|_{m} \}_{\lambda=1}^{r}$ forms a total basis for $\E_{m}$. In particular, one has $\gdim(\E_{m}) = \grk(\E)$. 
\item For every graded vector bundle map $F: \E \rightarrow \E'$ and every $m \in M$, there is a graded linear map $F_{m}: \E_{m} \rightarrow \E'_{m}$ determined uniquely by $F_{m}(\psi|_{m}) := F(\psi)|_{m}$ for all $\psi \in \Gamma_{\E}(M)$. 
\item For every $m \in M$ and $\ell \in \Z$, one has identifications $(\E^{\ast})_{m} \cong (\E_{m})^{\ast}$ and $(\E[\ell])_{m} \cong \E_{m}[\ell]$. 
\item Let $\cL \subseteq \E$ be a subbundle of $\E$ and let $I: \cL \rightarrow \E$ be the canonical inclusion. Then for each $m \in M$, $I_{m}: \cL_{m} \rightarrow \E_{m}$ is injective. Consequently, we always identify $\cL_{m}$ with the graded subspace $I_{m}(\cL_{m}) \subseteq \E_{m}$. 
\item Suppose that $(\E,g_{\E})$ is a quadratic graded vector bundle of degree $\ell$. For each $m \in M$, there is an induced degree $\ell$ isomorphism $g_{\E_{m}}: \E_{m} \rightarrow (\E_{m})^{\ast}$ and consequently a graded symmetric bilinear form $\<\cdot,\cdot\>_{\E_{m}}: \E_{m} \times \E_{m} \rightarrow \R$ of degree $\ell$. In other words, $(\E_{m}, g_{\E_{m}})$ becomes a quadratic graded vector space of degree $\ell$ for every $m \in M$.

Let $\cL \subseteq \E$ be a subbundle of $\E$. Then every fiber of its orthogonal complement is the orthogonal complement of the fiber, $(\cL^{\perp})_{m} = (\cL_{m})^{\perp}$ for all $m \in M$. 
\end{enumerate}
\end{tvrz}
See \S 5.2 of \cite{vysoky2021global} for details and proofs.
\begin{definice}
Let $(\E,g_{\E})$ be a quadratic graded vector bundle of degree $\ell$ over $\M$. Let $\cL \subseteq \E$ be its subbundle. We say that $\cL$ is an \textbf{isotropic subbundle}, if $\cL \subseteq \cL^{\perp}$. We say that an isotropic subbundle $\cL$ is \textbf{maximal}, if for every $m \in M$, the fiber $\cL_{m}$ is a maximal isotropic subspace of $\E_{m}$. We say that $\cL$ is \textbf{Lagrangian}, if $\cL = \cL^{\perp}$. 
\end{definice}
We find a graded vector bundle version of Proposition \ref{tvrz_maxisos1}. 
\begin{tvrz} \label{tvrz_maxisos2}
Let $(\E,g_{\E})$ be a quadratic vector bundle of degree $\ell$ over $\M$. Let $\cL \subseteq \E$ be its subbundle. Let $(r_{j})_{j \in \Z} = \grk(\E)$ and $(q_{j})_{j \in \Z} = \grk(\cL)$.
\begin{enumerate}[(1)]
\item For $\ell \pmod 4 \neq 0$, the following statements are equivalent: 
\begin{enumerate} 
\item $\cL$ is maximal isotropic;
\item $\cL$ is Lagrangian;
\item $\cL$ is isotropic and $r_{j} = q_{j} + q_{-(j+\ell)}$ for all $j \in \Z$. 
\end{enumerate}
\item Let $\ell \pmod 4 = 0$, that is $\ell = 2k$ for even $k \in \Z$. Then for each $m \in M$, $\<\cdot,\cdot\>_{\E_{m}}$ restricts to the non-degenerate symmetric bilinear form $\<\cdot,\cdot\>_{\E_{m}}^{0}$ on $(\E_{m})_{-k}$. Let $(s_{m},t_{m})$ be the signature of $\<\cdot,\cdot\>_{\E_{m}}^{0}$. Then the following statements are equivalent:
\begin{enumerate}
\item $\cL$ is maximal isotropic;
\item $\cL$ is Lagrangian (only valid when $s_{m} = t_{m}$ for all $m \in M$);
\item $\cL$ is isotropic, $q_{-k} = \min \{ s_{m},t_{m} \}$ and $r_{-k+i} = q_{-k+i} + q_{-k-i}$ for all $m \in M$ and $i \in \N$.
\end{enumerate}
\end{enumerate}
\end{tvrz}
\begin{rem}
All statements follow directly from Proposition \ref{tvrz_maxisos1}. Note that in the definition of isotropic subbundles, we assume that $\cL \subseteq \cL^{\perp}$, that is $\Gamma_{\cL} \subseteq \Gamma_{\cL^{\perp}}$, and not just the fiber-wise inclusion $\cL_{m} \subseteq (\cL_{m})^{\perp}$ for all $m \in M$. This is because by ``going fiber-wise'', one usually loses a lot of data. In other words, the isotropy of $\cL$ implies the isotropy of the graded vector space $\cL_{m}$ for all $m \in M$, but the converse statement may fail. Finally, observe that the signature $(s_{m},t_{m})$ is always constant in $m$ along each connected component of $M$. Since our subbundles are assumed to have a constant graded rank, we can safely assume that $(s_{m},t_{m})$ is globally constant (otherwise no maximally isotropic subbundles would exist).
\end{rem}
\begin{example}
Let $\E = T[\ell]\M \oplus T^{\ast}\M$ and $g_{\E}$ be defined as in Example \ref{ex_canpairing}. If $(n_{j})_{j \in \Z} = \gdim(\M)$, then $\grk(\E) = (n_{-(j+\ell)} + n_{j})_{j \in \Z}$. Recall that the \textbf{tangent space at $m \in M$} is defined as 
\begin{equation}
T_{m}\M = \gDer(\C^{\infty}_{\M}(M), \R).
\end{equation}
In other words, $v \in T_{m}\M$ is a degree $|v|$ graded linear map $v: \C^{\infty}_{\M}(M) \rightarrow \R$ satisfying 
\begin{equation}
v(fg) = v(f) g(m) + (-1)^{|v||f|} f(m) v(g),
\end{equation}
for all $f,g \in \C^{\infty}_{\M}(M)$. For each $m \in M$, $T_{m}\M$ can be canonically identified with the fiber $(T\M)_{m}$. For every vector field $X \in \X_{\M}(M)$, we have $X|_{m}(f) := (X(f))(m)$. Whence
\begin{equation}
\E_{m} = (T_{m}\M)[\ell] \oplus T^{\ast}_{m}\M.
\end{equation}
Let $(v,\alpha), (w,\beta) \in \E_{m}$. The induced bilinear form $\<(v,\alpha), (w,\beta)\>_{\E_{m}}$ is non-zero only if $|(v,\alpha)| + |(w,\beta)| + \ell = 0$. Recall that $|(v,\alpha)| = |v| - \ell = |\alpha|$, hence necessarily $|v| + |\beta| = 0$ and $|w| + |\alpha| = 0$, and in this case, one has 
\begin{equation}
\<(v,\alpha),(w,\beta)\>_{\E_{m}} = \alpha(w) + (-1)^{|v||w|} \beta(v). 
\end{equation}
Whenever $\ell = 2k$ for $k \in \Z$, we may restrict $\<\cdot,\cdot\>_{\E_{m}}$ to $(\E_{m})_{-k}$, finding an ordinary bilinear form
\begin{equation}
\<(v,\alpha),(w,\beta)\>_{\E_{m}}^{0} = \alpha(w) + (-1)^{k} \beta(v),
\end{equation}
for all $(v,\alpha),(w,\beta) \in (\E_{m})_{-k}$. For $\ell \pmod 4 = 0$, that is $k$ even, we see that $\<\cdot,\cdot\>^{0}_{\E_{m}}$ is indeed a symmetric bilinear form on $(\E_{m})_{-k}$ with a signature $(n_{-k},n_{-k})$, where $(n_{j})_{j \in \Z} = \gdim(\M)$. Hence by Proposition \ref{tvrz_maxisos2}, a subbundle $\cL \subseteq \E$ is maximally isotropic, iff it is Lagrangian, $\cL = \cL^{\perp}$. Equivalently, it has to be isotropic and its graded rank $(q_{j})_{j \in \Z}$ must for all $j \in \Z$ satisfy
\begin{equation} \label{eq_grkmaxisostandardgenvec}
n_{j} + n_{-(j+\ell)} = q_{j} + q_{-(j+\ell)}.
\end{equation}
\end{example}
\begin{definice}
Let $(\E,\rho,g_{\E},[\cdot,\cdot]_{\E})$ be a graded Courant algebroid of degree $\ell$. 

We say that a subbundle $\cL \subseteq \E$ is a \textbf{Dirac structure on $\E$}, if it is maximally isotropic and the graded submodule $\Gamma_{\cL}(M)$ is involutive under $[\cdot,\cdot]_{\E}$. 
\end{definice}
\begin{example} \label{ex_PoissonasDirac}
Let $(\E,\rho,g_{\E},[\cdot,\cdot]_{D}^{H})$ be a graded Courant algebroid of degree $\ell$ from Example \ref{ex_gDorfman}. This example should be considered a graded version of some ideas in the famous paper \cite{Severa:2001qm}.

Let us consider a vector bundle map $\Pi^{\sharp}: T^{\ast}\M \rightarrow T[\ell]\M$ of degree zero. Equivalently, this is a vector bundle map $\Pi^{\sharp}: T^{\ast}\M \rightarrow T\M$ of degree $\ell$. Its graph $\gr(\Pi^{\sharp})$ is a subbundle of $\E$ defined as
\begin{equation}
\Gamma_{\gr(\Pi^{\sharp})}(U) := \{ (\Pi^{\sharp}|_{U}(\xi),\xi) \; | \; \xi \in \Omega^{1}_{\M}(U) \} \subseteq \Gamma_{\E}(U),
\end{equation}
for each $U \in \Op(M)$. Observe that $\grk( \gr(\Pi^{\sharp})) = \grk(\Omega^{1}_{\M}) = (n_{j})_{j \in \Z}$. Hence it satisfies the constraint (\ref{eq_grkmaxisostandardgenvec}). The condition $\gr(\Pi^{\sharp}) \subseteq \gr(\Pi^{\sharp})^{\perp}$ is equivalent to
\begin{equation}
\<(\Pi^{\sharp}(\xi),\xi), (\Pi^{\sharp}(\eta),\eta) \>_{\E} = 0,
\end{equation}
for all $\xi,\eta \in \Omega^{1}_{\M}(M)$. By plugging in from (\ref{eq_canpairing}), this gives the equation
\begin{equation}
\xi(\Pi^{\sharp}(\eta)) + (-1)^{(|\xi|+\ell)(|\eta|+\ell)} \eta( \Pi^{\sharp}(\xi)) = 0.
\end{equation}
Using (\ref{eq_vfactionon1form}), one can rewrite this as $[\Pi^{\sharp}(\xi)](\eta) + (-1)^{|\xi||\eta|+\ell} [\Pi^{\sharp}(\eta)](\xi) = 0$. Let $\Pi: \X_{\M}(M) \times \X_{\M}(M) \rightarrow \C^{\infty}_{\M}(M)$ be a degree $\ell$ bilinear map defined by
\begin{equation} \label{eq_PiPisharp}
\Pi(\xi,\eta) := [\Pi^{\sharp}(\xi)](\eta),
\end{equation}
for all $\xi,\eta \in \Omega^{1}_{\M}(M)$. In other words, we have shown that $\gr(\Pi^{\sharp})$ is maximal isotropic, if and only if $\Pi^{\sharp}$ is induced using (\ref{eq_PiPisharp}) by a skew-symmetric bivector $\Pi \in \X^{2}_{\M}(M)$ for even $\ell$ and symmetric bivector $\Pi \in \~\X^{2}_{\M}(M)$ for odd $\ell$. Finally, for every $\xi,\eta \in \Omega^{1}_{\M}(M)$, one finds
\begin{equation}
\begin{split}
[(\Pi^{\sharp}(\xi),\xi), (\Pi^{\sharp}(\eta), \eta)]_{D}^{H} = & \ \big([\Pi^{\sharp}(\xi), \Pi^{\sharp}(\eta)], (-1)^{(|\xi|+\ell)\ell} \Li{\Pi^{\sharp}(\xi)}\eta \\
&  - (-1)^{(|\eta|+\ell)|\xi|} i_{\Pi^{\sharp}(\eta)} \dr{\xi} + (-1)^{\ell} H(\Pi^{\sharp}(\xi), \Pi^{\sharp}(\eta), \cdot) \big)
\end{split}
\end{equation}
It is convenient to define an $\R$-bilinear bracket $[\cdot,\cdot]_{\Pi}: \Omega^{1}_{\M}(M) \times \Omega^{1}_{\M}(M) \rightarrow \Omega^{1}_{\M}(M)$ by 
\begin{equation}
[\xi,\eta]_{\Pi} := (-1)^{(|\xi|+\ell)\ell} \Li{\Pi^{\sharp}(\xi)}\eta  - (-1)^{(|\eta|+\ell)|\xi|} i_{\Pi^{\sharp}(\eta)} \dr{\xi},
\end{equation}
for all $\xi,\eta \in \Omega^{1}_{\M}(M)$. Note that $|[\xi,\eta]_{\Pi}| = |\xi| + |\eta| + \ell$. We see that $\gr(\Pi^{\sharp})$ is involutive, iff 
\begin{equation}
[\Pi^{\sharp}(\xi), \Pi^{\sharp}(\eta)] = \Pi^{\sharp}\big( [\xi,\eta]_{\Pi} + (-1)^{\ell} H(\Pi^{\sharp}(\xi), \Pi^{\sharp}(\eta), \cdot) \big),
\end{equation}
for all $\xi,\eta \in \Omega^{1}_{\M}(M)$. This can be after a lot of work rewritten into the form
\begin{equation} \label{eq_PiPiSHtwisted}
\frac{1}{2}[\Pi,\Pi]_{S}(\xi,\eta,\zeta) = -(-1)^{(|\eta|+\ell)\ell} H(\Pi^{\sharp}(\xi), \Pi^{\sharp}(\eta), \Pi^{\sharp}(\zeta)),
\end{equation}
for all $\xi,\eta,\zeta \in \Omega^{1}_{\M}(M)$. Note that it is important that $\Pi \in \X^{2}_{\M}(M)$ for even $\ell$ and $\Pi \in \~\X^{2}_{\M}(M)$ for odd $\ell$. A derivation of this formula is straightforward but a bit tedious, hence we omit it here. We see that $\gr(\Pi^{\sharp}) \subseteq \E$ is a Dirac structure, iff $(\M,\Pi)$ defines an \textbf{$H$-twisted graded Poisson manifold of degree $\ell$}. For $H = 0$, it is just a graded Poisson manifold of degree $\ell$, see Example \ref{ex_gPoisson}. Recall that the bracket $\{\cdot,\cdot\}$ is obtained by $\{f,g\} = [[f,\Pi]_{S},g]_{S}$. Let 
\begin{equation}
\fJ(f,g,h) = \{f, \{g,h\}\} - \{\{f,g\},h\} - (-1)^{(|f|+\ell)(|g|+\ell)} \{g, \{f,h\}\}
\end{equation}
be the Jacobiator of the bracket $\{\cdot,\cdot\}$. Using Proposition \ref{tvrz_LASNB} and Proposition \ref{tvrz_LASNB2}, one can prove
\begin{equation} \label{eq_JacobiatorasSchouten}
\fJ(f,g,h) = \frac{1}{2} (-1)^{|f| + |g|(1+\ell) + |h|} [\Pi,\Pi]_{S}(\dr{f}, \dr{g}, \dr{h}). 
\end{equation}
Now, recall that for each $f \in \C^{\infty}_{\M}(M)$ the corresponding Hamiltonian vector field $X_{f} := [f,\Pi]_{S}$ can be written as $X_{f} = -(-1)^{|f|(1+\ell)} \Pi^{\sharp}(\dr{f})$. Combining (\ref{eq_PiPiSHtwisted}) and (\ref{eq_JacobiatorasSchouten}) then gives 
\begin{equation}
\fJ(f,g,h) = (-1)^{(|f|+|g|+|h|)\ell} H(X_{f},X_{g},X_{h}). 
\end{equation}
\end{example}
\section{Generalized complex structures}
In this section, we define a graded analogue of generalized complex structures, introduced in \cite{Gualtieri:2003dx}. 
Let $\E$ be a graded vector bundle over $\M$. Let $(\C_{\M}^{\infty})_{\Cn}$ denote the complexification of the structure sheaf $\C^{\infty}_{\M}$ of $\M$. Its sections over $U \in \Op(M)$ have the form $f + ig$, where $f,g \in \C^{\infty}_{\M}(U)$ satisfy $|f| = |g|$. The multiplication is defined in the usual way to make $(\C^{\infty}_{\M})_{\Cn}$ into the sheaf of \textit{complex} graded commutative associative algebras. 

Now, let $\Gamma_{\E_{\Cn}} := (\Gamma_{\E})_{\Cn}$ be the complexification of the sheaf of sections of $\E$. It follows that $\Gamma_{\E_{\Cn}}$ becomes a sheaf of \textit{complex} graded $(\C^{\infty}_{\M})_{\Cn}$-modules. We say that $\Gamma_{\E_{\Cn}}$ is a sheaf of sections of the \textit{complex} vector bundle $\E_{\Cn}$, called the \textbf{complexification of $\E$}. Sections in $\Gamma_{\E_{\Cn}}(M)$ are of the form $\psi + i \phi$, where $\psi,\phi \in \Gamma_{\E}(M)$ satisfy $|\psi| = |\phi|$, and for each $f + ig \in (\C^{\infty}_{\M})_{\Cn}(M)$, one has 
\begin{equation}
(f + ig)(\psi + i\phi) := f \psi - g \phi + i (f \phi + g \psi). 
\end{equation}
A theory of complex graded vector bundles and their subbundles is completely analogous to the real case, except that involved graded vector spaces are complex and all graded linear maps are $\Cn$-linear. Now, there is a canonical $(\C^{\infty}_{\M})_{\Cn}(M)$-antilinear map $\Sigma: \Gamma_{\E_{\Cn}}(M) \rightarrow \Gamma_{\E_{\Cn}}(M)$ of degree zero, defined for all $\psi + i \phi \in \Gamma_{\E_{\Cn}}(M)$ as 
\begin{equation}
\Sigma(\psi + i \phi) := \psi - i \phi,
\end{equation} 
$\Sigma$ is called the \textbf{complex conjugation on $\E_{\Cn}$}. Now, for any \textit{complex} subbundle $\cL \subseteq \E_{\Cn}$, let 
\begin{equation}
\Gamma_{\ol{\cL}}(U) := \Sigma|_{U}( \Gamma_{\cL}(U)).
\end{equation} 
It follows that $\Gamma_{\ol{\cL}}$ defines a sheaf of sections of a complex subbundle $\ol{\cL} \subseteq \E$ called the \textbf{complex conjugate of $\cL$}. One has $\grk_{\Cn}(\cL) = \grk_{\Cn}(\ol{\cL})$. 

Let $(\E,\rho,g_{\E},[\cdot,\cdot]_{\E})$ be a graded Courant algebroid of degree $\E$. All structures can be naturally complexified to obtain a \textit{complex} graded Courant algebroid $(\E_{\Cn}, \rho_{\Cn}, g_{\E_{\Cn}}, [\cdot,\cdot]_{\E_{\Cn}})$ of degree $\ell$, where $\rho_{\Cn}: \E_{\Cn} \rightarrow (T\M)_{\Cn}$ and $g_{\E_{\Cn}}: \E_{\Cn} \rightarrow (\E_{\Cn})^{\ast}$ are complex vector bundle morphisms and $[\cdot,\cdot]_{\E_{\Cn}}: \Gamma_{\E_{\Cn}}(M) \times \Gamma_{\E_{\Cn}}(M) \rightarrow \Gamma_{\E_{\Cn}}(M)$ is a $\Cn$-bilinear bracket. $g_{\E_{\Cn}}$ then induces a graded symmetric $\Cn$-bilinear form $\<\cdot,\cdot\>_{\E_{\Cn}}$.

\begin{definice}
Let $(\E,\rho,g_{\E},[\cdot,\cdot]_{\E})$ be a graded Courant algebroid of degree $\ell$. We say that a complex subbundle $\cL \subseteq \E_{\Cn}$ is a \textbf{generalized complex structure}, if 
\begin{enumerate}[(i)]
\item $\cL$ is isotropic with respect to $\<\cdot,\cdot\>_{\E_{\Cn}}$ and $\Gamma_{\cL}(M)$ is involutive with respect to $[\cdot,\cdot]_{\E_{\Cn}}$;
\item one can write $\E_{\Cn} = \cL \oplus \ol{\cL}$. 
\end{enumerate}
\end{definice}
\begin{rem}
Observe that contrary to the ordinary case, see e.g. Proposition 4.3 in \cite{Gualtieri:2003dx}, we do not assume that $\cL$ is a \textit{maximal} isotropic subbundle. However, let $(r_{j})_{j \in \Z} = \grk(\E) = \grk_{\Cn}(\E_{\Cn})$ and $(q_{j})_{j \in \Z} = \grk_{\Cn}(\cL)$. The condition (ii) forces $r_{j} = 2 q_{j}$ for each $j \in \Z$. Since $g_{\E}$ is an isomorphism, by Remark \ref{rem_nondegen} we have $r_{j} = r_{-(j+\ell)}$. Hence necessarily also $q_{j} = q_{-(j+\ell)}$. It follows that $\grk_{\Cn}(\cL) = \grk_{\Cn}(\cL^{\perp})$ and $\cL$ is automatically Lagrangian, hence maximal isotropic. 
\end{rem}
\begin{tvrz} \label{tvrz_Jgencomplex}
Let $(\E,\rho,g_{\E},[\cdot,\cdot]_{\E})$ be a graded Courant algebroid of degree $\ell$. Let $\cL \subseteq \E_{\Cn}$ be a generalized complex structure. Then it induces a degree zero vector bundle map $\J: \E \rightarrow \E$ having the following properties:
\begin{enumerate}[(i)]
\item $\J^{2} = -\1_{\E}$ and $\J$ is orthogonal with respect to $\<\cdot,\cdot\>_{\E}$;
\item The \textbf{Nijenhuis tensor} $N_{\J}$ defined for all $\psi,\psi' \in \Gamma_{\E}(M)$ by 
\begin{equation}
N_{\J}(\psi,\psi') := [\psi,\psi']_{\E} - [\J(\psi),\J(\psi')]_{\E} + \J([\psi,\J(\psi')]_{\E} + [\J(\psi),\psi']_{\E})
\end{equation}
vanishes identically. 
\item $\Gamma_{\cL}(U)$ coincides with the $+i$ eigenspace of $\J_{\Cn}|_{U}$ for each $U \in \Op(M)$. Consequently, $\Gamma_{\ol{\cL}}(U)$ coincides with the $-i$ eigenspace of $\J_{\Cn}|_{U}$ for each $U \in \Op(M)$. 
\end{enumerate}
\end{tvrz}
\begin{proof}
By assumption, one can decompose every $\psi \in \Gamma_{\E_{\Cn}}(M)$ as $\psi = \phi + \Sigma(\phi')$ for unique $\phi,\phi' \in \Gamma_{\cL}(M)$. Define $\J': \Gamma_{\E_{\Cn}}(M) \rightarrow \Gamma_{\E_{\Cn}}(M)$ as $\J'(\phi + \Sigma(\phi')) := i(\phi - \Sigma(\phi'))$. In other words, $\Gamma_{\cL}(M)$ and $\Gamma_{\ol{\cL}}(M)$ become the $+i$ and $-i$ eigenspaces of $\J'$, respectively. It is easy to see that
\begin{equation}
\J' \circ \Sigma = \Sigma \circ \J',
\end{equation}
whence $\J' = \J_{\Cn}$ for a unique degree zero $\C^{\infty}_{\M}(M)$-linear map $\J: \Gamma_{\E}(M) \rightarrow \Gamma_{\E}(M)$. $\J$ satisfies the property $(iii)$ by definition. Moreover, since clearly $\J'^{2} = -\1_{\E_{\Cn}}$ and $\J'$ is orthogonal with respect to $\<\cdot,\cdot\>_{\E_{\Cn}}$, the properties $(i)$ of $\J$ follow immediately. 

It follows from the construction of $\J$ that $\Gamma_{\cL}(M)$ is involutive with respect to $[\cdot,\cdot]_{\E_{\Cn}}$, iff 
\begin{equation}
[\psi - i\J(\psi), \psi' - i \J(\psi')]_{\E_{\Cn}} \in \Gamma_{\cL}(M),
\end{equation}
for all $\psi,\psi' \in \Gamma_{\E}(M)$. The projection of this expression to $\Gamma_{\ol{\cL}}(M)$ must vanish, that is
\begin{equation}
\begin{split}
0 = & \ (\1_{\E_{\Cn}} + i \J_{\Cn})( [\psi - i\J(\psi), \psi' - i \J(\psi')]_{\E_{\Cn}}) \\
= & \ N_{\J}(\psi,\psi') + i \J( N_{\J}(\psi,\psi')).
\end{split}
\end{equation}
We see that $\Gamma_{\cL}(M)$ is involutive with respect to $[\cdot,\cdot]_{\E_{\Cn}}$, iff $N_{\J}(\psi,\psi') = 0$ for all $\psi,\psi' \in \Gamma_{\E}(M)$. This proves the claim $(ii)$. 
\end{proof}
\begin{rem} \label{rem_gencomplex}
For ordinary Courant algebroids, every vector bundle map $\J: \E \rightarrow \E$ having the above properties $(i)$ and $(ii)$ determines the unique generalized complex structure $\cL \subseteq \E_{\Cn}$ by declaring $\cL$ to satisfy $(iii)$. For \emph{graded} Courant algebroids, it may happen that $\Gamma_{\cL}$ defined by
\begin{equation} \label{eq_gcLintermsofJ}
\Gamma_{\cL} = \ker(\J_{\Cn} - i \1_{\E_{\Cn}})
\end{equation} 
fails to define a complex subbundle $\cL$ of $\E_{\Cn}$. However, if one \textit{adds} the assumption that $\Gamma_{\cL}$ defines a complex subbundle of $\E_{\Cn}$, it is not difficult to argue that $\cL$ is a generalized complex structure based on the properties $(i)$ of $(ii)$ of $\J$. Note that this technical issue is closely related to the one described in Remark \ref{rem_noSerreswan}. 
\end{rem}
\begin{example} \label{ex_gcomplexsymplectic}
Let $(\E,\rho,g_{\E},[\cdot,\cdot]_{D})$ be the graded Courant algebroid of degree $\ell$ from Example \ref{ex_gDorfman} with $H = 0$. Let $\omega \in \Omega^{2}_{\M}(M)$ be a $2$-form of degree $-\ell$. Let $\omega^{\flat}: T\M \rightarrow T^{\ast}\M$ be the degree $\ell$ vector bundle map induced by $\omega$, that is 
\begin{equation}
[\omega^{\flat}(X)](Y) := \omega(X,Y),
\end{equation}
for all $X,Y \in \X_{\M}(M)$. $\omega^{\flat}$ can be viewed as a vector bundle map $\omega^{\flat}: T[\ell]\M \rightarrow T^{\ast}\M$ of degree zero. Suppose that $\omega$ is a \textbf{symplectic form}, that is $\omega^{\flat}$ is an isomorphism and $\dr{\omega} = 0$. 

For each $U \in \Op(M)$, let us define the following graded $(\C^{\infty}_{\M}(U))_{\Cn}$-submodule of $\Gamma_{\E_{\Cn}}(U)$:
\begin{equation}
\Gamma_{\cL}(U) := \{ (X, \omega^{\flat}|_{U}(Y)) + i(Y, -\omega^{\flat}|_{U}(X)) \; | \; X+iY \in (\X_{\M}(U))_{\Cn} \}
\end{equation}
It is straightforward to verify that $\Gamma_{\cL}$ defines a sheaf of sections of a complex subbundle $\cL \subseteq \E_{\Cn}$. By construction, this subbundle is isomorphic to $(T\M)_{\Cn}$. It is straightforward to verify that $\E_{\Cn} = \cL \oplus \ol{\cL}$ and $\cL \subseteq \cL^{\perp}$. Finally, its involutivity can be shown to be equivalent to $\dr{\omega} = 0$. In other words, $\cL$ defines a generalized complex structure. The corresponding vector bundle morphism $\J: \E \rightarrow \E$ takes the explicit form 
\begin{equation} \label{eq_Jforgencomplexsymplectic}
\J(X,\xi) = (-(\omega^{\flat})^{-1}(\xi), \omega^{\flat}(X)),
\end{equation}
for all $(X,\xi) \in \Gamma_{\E}(M)$. Note that for $H \neq 0$, $\cL$ is \textit{not} involutive for any $\omega$. 
\end{example}
\begin{example} 
Let $(\E,\rho,g_{\E},[\cdot,\cdot]_{D}^{H})$ be the graded Courant algebroid of degree $\ell$ from Example \ref{ex_gDorfman}. Let $J: T\M \rightarrow T\M$ be a degree zero vector bundle map. We say that $J$ is a \textbf{complex structure} on $\M$, if it has the following properties:
\begin{enumerate}[(i)]
\item $J^{2} = -\1_{T\M}$.
\item Let $J_{\Cn}: (T\M)_{\Cn} \rightarrow (T\M)_{\Cn}$ be the complexification of $J$, and define
\begin{equation}
\X^{1,0}_{\M} := \ker( J_{\Cn} - i \1_{(T\M)_{\Cn}}).
\end{equation} 
Then $\X^{1,0}_{\M}$ must define a sheaf of sections of a complex subbundle $T^{1,0}\hspace{-0.8mm}\M$ of $(T\M)_{\Cn}$, such that $\X^{1,0}_{\M}(M)$ is involutive with respect to the complexified graded commutator $[\cdot,\cdot]_{\Cn}$. 
\end{enumerate}
For ordinary manifolds, $\X^{1,0}_{\M}$ defines a subbundle for any $J$ satisfying (i). However, in the graded setting, this is no longer true. The involutivity of $\X^{1,0}_{\M}(M)$ is equivalent to 
\begin{equation}
0 = N_{J}(X,Y) = [X,Y] - [J(X),J(Y)] + J([X,J(Y)] + [J(X),Y]),
\end{equation}
for all $X,Y \in \X_{\M}(M)$. The complex conjugate to $\X^{1,0}_{\M}$ is $\X^{0,1}_{\M} = \ker(J_{\Cn} + i \1_{(TM)_{\Cn}})$, and let
\begin{equation}
\Omega^{1,0}_{\M} := \an( \X^{0,1}_{\M}), \; \; \Omega^{0,1}_{\M} := \an( \X^{1,0}_{\M}). 
\end{equation} 
It follows easily that $(T\M)_{\Cn} = \X^{1,0}_{\M} \oplus \X^{0,1}_{\M}$ and $(T^{\ast}\M)_{\Cn} = \Omega^{1,0}_{\M} \oplus \Omega^{0,1}_{\M}$. Let 
\begin{equation}
\Gamma_{\cL} := \X^{1,0}_{\M}[\ell] \oplus \Omega^{0,1}_{\M}.
\end{equation}
Now, the complexified Courant algebroid $\E_{\Cn}$ can be identified with $\E_{\Cn} = (T\M)_{\Cn}[\ell] \oplus (T^{\ast}\M)_{\Cn}$. The complexified fiber-wise metric and the bracket then take the same form as in Example \ref{ex_gDorfman}, except that all operations are replaced by their complexified versions. In particular, the bracket contains the complexification $H_{\Cn}$ of the $3$-form $H$. 

It is then easy to see that $\Gamma_{\cL}(M)$ is involutive, iff $\X^{1,0}_{\M}(M)$ is involutive with respect to $[\cdot,\cdot]_{\Cn}$ (this follows from the property (ii) of complex structures), and 
\begin{equation}
H_{\Cn}(X,Y,Z) = 0,
\end{equation}
for all $X,Y,Z \in \X_{\M}^{1,0}(M)$. In terms of $H$ and $J$, this condition can be rewritten as 
\begin{equation}
H(X,Y,Z) = H(J(X),J(Y),Z) + H(X,J(Y),J(Z)) + H(J(X),Y,J(Z)),
\end{equation}
for all $X,Y,Z \in \X_{\M}(M)$. Finally, the corresponding vector bundle map $\J: \E \rightarrow \E$ takes the form
\begin{equation}
\J(X,\xi) = (J(X), -J^{T}(\xi)),
\end{equation}
for all $(X,\xi) \in \Gamma_{\E}(M)$. Note that although we call $J$ a complex structure on $\M$, we do not claim that there is a graded version of the Newlander--Nirenberg theorem. 
\end{example}
\section{Differential graded Courant algebroids}
A grading of a space of global sections of a graded vector bundle $\E$ allows one to consider maps of a non-zero degree. Suppose $(\E,\rho,g_{\E},[\cdot,\cdot]_{\E})$ is a graded Courant algebroid of degree $\ell \in \Z$. 

In particular, we can consider a degree $1$ derivation of the bracket $[\cdot,\cdot]_{\E}$, that is an $\R$-linear map $\Delta: \Gamma_{\E}(M) \rightarrow \Gamma_{\E}(M)$ of degree $1$ satisfying 
\begin{equation} \label{eq_Deltaderivation}
\Delta([\psi,\psi']_{\E}) = [\Delta(\psi),\psi']_{\E} + (-1)^{|\psi|+\ell} [\psi, \Delta(\psi')]_{\E}, 
\end{equation}
for all $\psi,\psi' \in \Gamma_{\E}(M)$. However, $\Gamma_{\E}(M)$ is not just a graded vector space, but a graded $\C^{\infty}_{\M}(M)$-module. Moreover, the bracket $[\cdot,\cdot]_{\E}$ is not skew-symmetric. We thus have to make sure that the above condition is consistent with the axioms of a graded Courant algebroid. Taking this into consideration, one ends up with the following definition. 

\begin{definice} \label{def_delta}
Let $(\E,\rho,g_{\E},[\cdot,\cdot]_{\E})$ be a graded Courant algebroid of degree $\ell$. Let $\Delta: \Gamma_{\E}(M) \rightarrow \Gamma_{\E}(M)$ be a degree $1$ graded $\R$-linear map, such that:
\begin{enumerate}[(i)]
\item $\Delta$ squares to zero, that is $\Delta^{2} = 0$.
\item $\Delta$ is compatible with the graded $\C^{\infty}_{\M}(M)$-module structure on $\Gamma_{\E}(M)$, that is there exists a vector field $\ul{\Delta} \in \X_{\M}(M)$, such that 
\begin{equation} \label{eq_DeltaLeibniz}
\Delta(f\psi) = (-1)^{|f|} f \Delta(\psi) + \ul{\Delta}(f) \psi,
\end{equation}
 for all $f \in \C^{\infty}_{\M}(M)$ and $\psi \in \Gamma_{\E}(M)$.
\item $\Delta$ is compatible with the fiber-wise metric, that is for all $\psi,\psi \in \Gamma_{\E}(M)$, one has 
\begin{equation} \label{eq_Deltametric}
\ul{\Delta} \<\psi,\psi'\>_{\E} = \< \Delta(\psi), \psi'\>_{\E} + (-1)^{|\psi|+\ell} \<\psi, \Delta(\psi')\>_{\E},
\end{equation}
\item $\Delta$ is compatible with the anchor, that is for all $\psi \in \Gamma_{\E}(M)$, one has 
\begin{equation} \label{eq_Deltaanchor}
\rho( \Delta(\psi)) = (-1)^{\ell} [\ul{\Delta}, \rho(\psi)],
\end{equation}
\item $\Delta$ is a graded derivation of the bracket $[\cdot,\cdot]_{\E}$, that is it satisfies (\ref{eq_Deltaderivation}) for all $\psi,\psi' \in \Gamma_{\E}(M)$. 
\end{enumerate}
We say that $\Delta$ is a \textbf{differential on $\E$}, and $(\E,\rho,g_{\E},[\cdot,\cdot]_{\E},\Delta)$ is called a \textbf{differential graded Courant algebroid of degree $\ell$}. 
\end{definice}
\begin{rem}
For any $\Delta$ satisfying (\ref{eq_Deltametric}), one can show that (\ref{eq_Deltaanchor}) is equivalent to 
\begin{equation}
\Delta \circ \dD + \dD \circ \ul{\Delta} = 0.
\end{equation}
Note that the axioms above can be modified to work for $\Delta$ of any given degree $|\Delta|$. However, to keeps things simple, we only consider $|\Delta| = 1$. 
\end{rem}
\begin{example} \label{ex_homologicalsection}
Let $\phi \in \Gamma_{\E}(M)$ be a given section of degree $|\phi| = 1 - \ell$. Then $\Delta = [\phi,\cdot]_{\E}$ is a graded $\R$-linear map of degree $1$. Let $\ul{\Delta} := \hat{\rho}(\phi) \equiv (-1)^{\ell} \rho(\phi)$. It is easy to see that (\ref{eq_Deltaderivation} - \ref{eq_Deltaanchor}) follow immediately from the axioms (c1)-(c4) in Definition \ref{def_gCourant} and the equation (\ref{eq_rhoishom}). 

It only remains to verify when $\Delta^{2} = 0$. Using the graded Jacobi identity, one finds
\begin{equation} \label{eq_Deltasquareforphi}
\Delta^{2}(\psi) = \frac{1}{2} [[\phi,\phi]_{\E}, \psi]_{\E}, 
\end{equation}
for all $\psi \in \Gamma_{\E}(M)$. This does not necessarily imply that $[\phi,\phi]_{\E} = 0$. For example, it follows from (\ref{eq_bracketswithD}) that it suffices to ensure that $[\phi,\phi]_{\E} = \dD{f}$ for some $f \in \C^{\infty}_{\M}(M)$ with $|f| = 2 - \ell$. We say that $\phi$ is a \textbf{homological section}, if it satisfies $[[\phi,\phi]_{\E},\psi]_{\E} = 0$ for all $\psi \in \Gamma_{\E}(M)$. 

Note that in view of the above remark, (\ref{eq_Deltasquareforphi}) is true only if $|\Delta|$ is odd. 
\end{example}

\begin{example} \label{ex_Deltaforexact}
Let $(\E,\rho,g_{\E},[\cdot,\cdot]_{D}^{H})$ be a graded Courant algebroid of degree $\ell$ from Example \ref{ex_gDorfman}. Suppose $\Delta$ makes it into a differential graded Courant algebroid. Let $Q := (-1)^{\ell} \ul{\Delta}$. It follows from the properties (ii)-(iv) of Definition \ref{def_delta} that $\Delta$ has to have the explicit form
\begin{equation} \label{eq_Deltaforexact}
\Delta(X,\xi) = ([Q,X], (-1)^{\ell}\{ \Li{Q}\xi + \theta^{\flat}(X) + H(Q,X,\cdot)\}),
\end{equation}
for all $(X,\xi) \in \Gamma_{\E}(M)$. The vector bundle map $\theta^{\flat}: T\M \rightarrow T^{\ast}\M$ is induced by a $2$-form $\theta \in \Omega^{2}_{\M}(M)$ of degree $|\theta| = 1 - \ell$, that is $[\theta^{\flat}(X)](Y) := \theta(X,Y)$ for all $X,Y \in \X_{\M}(M)$. The equation (\ref{eq_Deltaderivation}) is then equivalent to $\dr{\theta} = 0$. 

Now, observe that for $\ell \neq 1$, $\theta = \dr{\vartheta}$ for some $\vartheta \in \Omega^{1}_{\M}(M)$ of degree $|\vartheta| = 1 - \ell$. See \S 6.5 in \cite{vysoky2021global} for details. In this case, (\ref{eq_Deltaforexact}) can be for all $(X,\xi) \in \Gamma_{\E}(M)$ rewritten simply as 
\begin{equation}
\Delta(X,\xi) = [(Q,\vartheta), (X,\xi)]_{D}^{H}. 
\end{equation}
In other words, every possible $\Delta$ takes the form as in Example \ref{ex_homologicalsection} for a homological section $\phi = (Q,\vartheta)$. However, for $\ell = 1$, $\theta$ is in general not exact. 

It remains to examine the condition $\Delta^{2} = 0$. First, one finds the equation $[Q,[Q,X]] = 0$ for all $X \in \X_{\M}(M)$. This can be rewritten as $[[Q,Q],X] = 0$ for all $X \in \X_{\M}(M)$ and, similarly to ordinary manifolds, this is equivalent $[Q,Q] = 0$. Having this in mind, the only non-trivial remaining condition becomes
\begin{equation}
\Li{Q}( \theta + i_{Q}H) = 0.
\end{equation}
Taking into account that $\theta$ and $H$ are closed, this can be further rewritten as 
\begin{equation}
\dr( i_{Q}\theta + \frac{1}{2} i_{Q}i_{Q}H) = 0. 
\end{equation}
Note that if $\ell \neq 2$, this is equivalent to $i_{Q}\theta + \frac{1}{2} i_{Q}i_{Q}H = \dr{f}$ for some $f \in \C^{\infty}_{\M}(M)$ of degree $|f| = 2 - \ell$. We see that in any case, $\M$ is equipped with a \textbf{homological vector field} $Q$, thus making $(\M,Q)$ into a differential graded manifold. 
\end{example}

Having a notion of a differential $\Delta$, it makes sense to impose some its compatibility properties with additional structures on $\E$. 
\begin{definice}
Let $(\E,\rho,g_{\E},[\cdot,\cdot]_{\E},\Delta)$ be a differential graded Courant algebroid of degree $\ell$. Let $\cL \subseteq \E$ be a Dirac structure on $\E$. We say that $\cL$ is \textbf{$\Delta$-compatible}, if $\Delta( \Gamma_{\cL}(M)) \subseteq \Gamma_{\cL}(M)$. 
\end{definice}
\begin{example}
Consider the scenario of Example \ref{ex_Deltaforexact}. In particular, $\Delta$ is given by (\ref{eq_Deltaforexact}), where $Q \in \X_{\M}(M)$ is a degree $|Q| = 1$ homological vector field and $\theta \in \Omega^{2}_{\M}(M)$ is a closed $2$-form of degree $|\theta| = 1 - \ell$. Let $\cL = \gr(\Pi^{\sharp})$ as in Example \ref{ex_PoissonasDirac}. In particular, $\Pi$ is an $H$-twisted graded Poisson structure of degree $\ell$. 

The condition $\Delta( \Gamma_{\gr(\Pi^{\sharp})}(M)) \subseteq \Gamma_{\gr(\Pi^{\sharp})}(M)$ is equivalent to 
\begin{equation}
(\Li{Q}\Pi)(\xi,\eta) + (-1)^{\ell(1+|\xi|)} (\theta + i_{Q}H)(\Pi^{\sharp}(\xi), \Pi^{\sharp}(\eta)) = 0.
\end{equation}
In terms of the Poisson bracket $\{f,g\} = [[f,\Pi]_{S},g]_{S}$ and Hamiltonian vector fields $X_{f} = [f,\Pi]_{S}$, this can be rewritten as 
\begin{equation}
Q\{f,g\} = \{Q(f),g\} + (-1)^{|f|+\ell} \{f, Q(g)\} + (-1)^{\ell(1+|f|+|g|)} (\theta + i_{Q}H)(X_{f},X_{g}),
\end{equation}
for all $f,g \in \C^{\infty}_{\M}(M)$. We see that this example provides a more general framework for so called QP-manifolds \cite{Schwarz:1992nx, Ikeda:2011ax}, in particular for their $H$-twisted version. 
\end{example}
A similar definition can be cooked for generalized complex structures.
\begin{definice}
Let $(\E,\rho,g_{\E},[\cdot,\cdot]_{\E},\Delta)$ be a differential graded Courant algebroid of degree $\ell$. Let $\cL \subseteq \E_{\Cn}$ be a generalized complex structure on $\E$. We say that $\cL$ is \textbf{$\Delta$-compatible}, if $\Delta_{\Cn}( \Gamma_{\cL}(M)) \subseteq \Gamma_{\cL}(M)$, where $\Delta_{\Cn}: \Gamma_{\E_{\Cn}}(M) \rightarrow \Gamma_{\E_{\Cn}}(M)$ is a complexification of $\Delta$. 

If $\J: \E \rightarrow \E$ corresponds to $\cL$, the above condition is equivalent to $\Delta \circ \J = \J \circ \Delta$. 
\end{definice}
\begin{example}
Suppose we have a setting as in Example \ref{ex_Deltaforexact}. In particular, we have $\Delta$ given by (\ref{eq_Deltaforexact}) for a homological vector field $Q$ of degree $|Q|=1$ and closed $2$-form $\theta$ of degree $|\theta| = 1 - \ell$. Let $\cL \subseteq \E_{\Cn}$ be the generalized complex structure from Example \ref{ex_gcomplexsymplectic}, that is the one determined by a symplectic form $\omega \in \Omega^{2}_{\M}(M)$ of degree $-\ell$. $\J$ is then given by (\ref{eq_Jforgencomplexsymplectic}). Note that we have to assume that $H = 0$. 

The condition $\Delta \circ \J = \J \circ \Delta$ immediately forces $\theta = 0$ and $\Li{Q}(\omega) = 0$. We see that $\cL$ is $\Delta$-compatible, if and only if $Q$ is symplectic with respect to $\omega$ and $\theta = 0$. In other word, $(\M,Q,\omega)$ is a \textbf{differential graded symplectic manifold} (sometimes also called a QP-manifold).
\end{example}
\section{Graded Lie bialgebroids}
Standard Courant algebroids originally rose to prominence as Manin triples for Lie bialgebroids \cite{liu1997manin}. In this section, we show how this procedure can be generalized to the graded setting. First, let us introduce a notion of graded Lie algebroid. 
\begin{definice}
$(\A, \da, [\cdot,\cdot]_{\A})$ is a called a \textbf{graded Lie algebroid of degree $\ell$}, if
\begin{enumerate}[(i)]
\item $\A$ is a graded vector bundle over a graded manifold $\M$;
\item $\da: \A \rightarrow T\M$ is a vector bundle map of degree $\ell$ called the anchor;
\item $[\cdot,\cdot]_{\A}: \Gamma_{\A}(M) \times \Gamma_{\A}(M) \rightarrow \Gamma_{\A}(M)$ is an $\R$-bilinear map of degree $\ell$, that is $|[X,Y]_{\A}| = |X| + |Y| + \ell$ for all $X,Y \in \Gamma_{\A}(M)$. 
\end{enumerate}
Moreover, the operations are subject to the following set of axioms. All conditions have to hold for all involved sections and functions. Let $\hat{\da}(X) := (-1)^{(|X|+\ell)\ell} \da(X)$. 
\begin{enumerate}[(c1)]
\item The bracket satisfies the \textbf{graded Leibniz rule} 
\begin{equation} \label{eq_LALeibniz}
[X, fY]_{\A} = \hat{\da}(X)(f) Y + (-1)^{|f|(|X|+\ell)} f [X,Y]_{\A}.
\end{equation}
\item The bracket is graded skew-symmetric, that is 
\begin{equation}
[X,Y]_{\A} = - (-1)^{(|X|+\ell)(|Y|+\ell)} [Y,X]_{\A}. 
\end{equation}
\item The bracket satisfies the \textbf{graded Jacobi identity}
\begin{equation} \label{eq_LAgJI}
[X,[Y,Z]_{\A}]_{\A} = [[X,Y]_{\A}, Z]_{\A} + (-1)^{(|X|+\ell)(|Y|+\ell)} [Y, [X, Z]_{\A}]_{\A}.
\end{equation}
\end{enumerate}
In other words, $(\Gamma_{\A}(M), [\cdot,\cdot]_{\A})$ becomes a graded Lie algebra of degree $\ell$. 
\end{definice}
Similarly to graded Courant algebroids, see Proposition \ref{tvrz_CAconsequences}, one can use (\ref{eq_LALeibniz}, \ref{eq_LAgJI}) to prove 
\begin{equation}
\da( [X,Y]_{\A}) = [\da(X),\da(Y)],
\end{equation}
for all $X,Y \in \Gamma_{\A}(M)$. 
\begin{example} \label{ex_standardgLA}
A standard example of a graded Lie algebroid of degree $\ell$ is $(T[\ell]\M, \1_{T\M}, [\cdot,\cdot])$, where the identity is viewed as a vector bundle map $\1_{T\M}: T[\ell]\M \rightarrow T\M$ of degree $\ell$, and $[\cdot,\cdot]$ is the graded commutator (\ref{eq_gcommutator}) of vector fields. 
\end{example}
Now, every graded Lie algebroid induces a differential $\dr_{\A}$ and two different Schouten-Nijenhuis brackets $[\cdot,\cdot]_{\A}$. In order to save space, we have decided to move their definitions and basic properties into the appendix \ref{sec_appendixA}. One can now use them to define a concept of a graded Lie bialgebroid. This generalizes the notion introduced in \cite{mackenzie1994} and it uses the approach taken in \cite{kosmann1995exact}.
\begin{definice}
Let $(\A,\da,[\cdot,\cdot]_{\A})$ be a graded Lie algebroid of degree $\ell$ and $(\A^{\ast}, \da', [\cdot,\cdot]_{\A^{\ast}})$ a graded Lie algebroid of degree $\ell'$, where $\A^{\ast}$ is the dual of $\A$. 

We say that $(\A,\A^{\ast})$ is a \textbf{graded Lie bialgebroid of bidegree $(\ell,\ell')$} if the induced differential $\dr_{\A}$ is a graded derivation of the Schouten-Nijenhuis bracket $[\cdot,\cdot]_{\A^{\ast}}$. In other words:
\begin{enumerate}[(i)]
\item Suppose $\ell$ is even. For all $\xi \in \Omega^{p}_{[\A]}(M) \cong \X^{p}_{[\A]}(M)$ and $\eta \in \Omega^{q}_{[\A]}(M) \cong \X^{q}_{[\A]}(M)$, one has 
\begin{equation} \label{eq_GLBAduality_even}
\dr_{\A} [\xi,\eta]_{\A} = [ \dr_{\A}\xi, \eta]_{\A^{\ast}} + (-1)^{|X|+\ell'+p-1} [\xi, \dr_{\A}\eta]_{\A^{\ast}}.
\end{equation}
\item Suppose $\ell$ is odd. For all $\xi \in \~\Omega^{p}_{[\A]}(M) \cong \~\X^{p}_{[\A]}(M)$ and $\eta \in \~\Omega^{q}_{[\A]}(M) \cong \~\X^{q}_{[\A]}(M)$, one has 
\begin{equation} \label{eq_GLBAduality_odd}
\dr_{\A} [\xi,\eta]_{\A} = [ \dr_{\A}\xi, \eta]_{\A^{\ast}} + (-1)^{|X|+\ell'} [\xi, \dr_{\A}\eta]_{\A^{\ast}}.
\end{equation}
\end{enumerate}
\end{definice}
Since $\dr_{\A}$ satisfies the graded Leibniz rule, it suffices to verify (\ref{eq_GLBAduality_even}) and (\ref{eq_GLBAduality_odd}) for $p,q \in \{0,1\}$. Let us now summarize those conditions.
\begin{tvrz}
The condition (\ref{eq_GLBAduality_even}) (for even $\ell$) and (\ref{eq_GLBAduality_odd}) (for odd $\ell$) is equivalent to the following set of conditions:
\begin{enumerate}[(i)]
\item For every $f \in \C^{\infty}_{\M}(M)$, one has 
\begin{equation} \label{eq_GLBAduality00}
\hat{\da}(\dr_{\A^{\ast}}f) - (-1)^{(\ell+1)(\ell'+1)} \hat{\da}'( \dr_{\A}f) = 0. 
\end{equation}
\item For every $f \in \C^{\infty}_{\M}(M)$ and $\xi \in \Omega^{1}_{[\A]}(M)$, one has 
\begin{equation} \label{eq_GLBAduality01}
\Li{\dr_{\A^{\ast}}f}^{\A}(\xi) - (-1)^{(\ell+1)(\ell'+1)} [\dr_{\A}f,\xi]_{\A^{\ast}} = 0.
\end{equation}
\item For every $\xi,\eta \in \Omega^{1}_{[\A]}(M)$, one has 
\begin{equation} \label{eq_GLBAduality11}
\dr_{\A} [\xi,\eta]_{\A^{\ast}} = (-1)^{|\xi|+\ell'} \Li{\xi}^{\A^{\ast}}( \dr_{\A}\eta) - (-1)^{(|\eta|+\ell')(|\xi|+\ell'+1)} \Li{\eta}^{\A^{\ast}}( \dr_{\A}\xi). 
\end{equation}
\end{enumerate}
\end{tvrz}
\begin{proof}
This follows immediately from definitions. Observe that (\ref{eq_GLBAduality11}) is the equation of objects in $\Omega^{2}_{[\A]}(M)$ for even $\ell$, and in $\~\Omega^{2}_{[\A]}(M)$ for odd $\ell$. 
\end{proof}
Similarly to the ordinary case, one obtains the following important observation.
\begin{tvrz} \label{tvrz_gLBAdual}
Let $(\A,\A^{\ast})$ be a graded Lie bialgebroid of bidegree $(\ell,\ell')$. 

Then $(\A^{\ast},\A)$ is a graded Lie bialgebroid of bidegree $(\ell',\ell)$. 
\end{tvrz}
\begin{proof}
The condition (\ref{eq_GLBAduality00}) is manifestly symmetric in $(\A,\ell) \leftrightarrow (\A^{\ast},\ell')$. The evaluation of (\ref{eq_GLBAduality01}) on $X \in \Gamma_{\A}(M)$ and using (\ref{eq_GLBAduality00}) gives the condition $(\ref{eq_GLBAduality01})$ for $(\A^{\ast},\A)$ evaluated on $\xi \in \Gamma_{\A^{\ast}}(M)$. Finally, the evaluation of (\ref{eq_GLBAduality11}) on $X,Y \in \Gamma_{\A}(M)$ and using (\ref{eq_GLBAduality00}, \ref{eq_GLBAduality01}) gives the condition (\ref{eq_GLBAduality11}) for $(\A^{\ast},\A)$ evaluated on $\xi,\eta \in \Gamma_{\A^{\ast}}(M)$. This is an analogue of Theorem 3.10 in \cite{mackenzie1994}, except a lot of signs is involved. We \textit{do not} encourage the reader to verify this statement. 
\end{proof}
It turns out a graded Lie bialgebroid can be used to construct a graded Courant algebroid.
\begin{theorem} \label{thm_doubleofgLBA}
Let $(\A,\A^{\ast})$ be a graded Lie bialgebroid of bidegree $(\ell,\ell')$ over a graded manifold $\M$. Consider the following data:
\begin{enumerate}[(i)]
\item Let $\E := \A[\ell'] \oplus \A^{\ast}[\ell]$. 
\item Define $\rho: \E \rightarrow T\M$ as $\rho(X,\xi) := \da(X) + \da'(\xi)$ for all $(X,\xi) \in \Gamma_{\E}(M)$. 
\item Let $g_{\E}: \E \rightarrow \E^{\ast}$ be for any $(X,\xi) \in \Gamma_{\E}(M)$ given by $g_{\E}(X,\xi) := (\xi, (-1)^{\ell \ell'} X)$, where we use the identification $\E^{\ast} \cong \A^{\ast}[-\ell'] \oplus \A[-\ell]$. 
\item Consider the $\R$-bilinear bracket 
\begin{equation}
\begin{split}
[(X,\xi),(Y,\eta)]_{\E} = \big( & [X,Y]_{\A} + (-1)^{(|X|+\ell)(\ell+\ell')} \Li{\xi}^{\A^{\ast}} Y - (-1)^{|X|+\ell'} (\dr_{\A^{\ast}}X)(\eta,\cdot), \\
& [\xi,\eta]_{\A^{\ast}} + (-1)^{(|X|+\ell)(\ell+\ell')} \Li{X}^{\A} \eta - (-1)^{|X|+\ell'} (\dr_{\A}\xi)(Y,\cdot) \big).
\end{split}
\end{equation}
for all $(X,\xi),(Y,\eta) \in \Gamma_{\E}(M)$. 
\end{enumerate}
Then $(\E,\rho,g_{\E},[\cdot,\cdot]_{\E})$ becomes a graded Courant algebroid of degree $\ell + \ell'$ called a \textbf{double of a Lie bialgebroid $(\A,\A^{\ast})$}. Both $\A[\ell']$ and $\A^{\ast}[\ell]$ are Dirac structures on $\E$. 
\end{theorem}
\begin{proof}
First, observe that for any $(X,\xi) \in \Gamma_{\E}(M)$, one has $|(X,\xi)| = |X| - \ell' = |\xi| - \ell$, where $|X|$ and $|\xi|$ always denote the original degrees in $\Gamma_{\A}(M)$ and $\Gamma_{\A^{\ast}}(M)$, respectively. Then
\begin{equation}
|\da(X)| = |X| + \ell = |(X,\xi)| + \ell + \ell', \; \; |\da'(\xi)| = |\xi| + \ell' = |(X,\xi)| + \ell + \ell'.
\end{equation}
This shows that $\rho$ is indeed a graded vector bundle map of degree $\ell + \ell'$. Similarly, for any $(\xi,X) \in \Gamma_{\E^{\ast}}(M)$, one has $|(\xi,X)| = |\xi| + \ell' = |X| + \ell$, whence 
\begin{equation}
|g_{\E}(X,\xi)| = |(\xi,(-1)^{\ell \ell'}X)| = |\xi| + \ell' = |(X,\xi)| + \ell + \ell'.
\end{equation}
This shows that $g_{\E}$ is a graded vector bundle isomorphism of degree $\ell + \ell'$. The corresponding fiber-wise metric $\<\cdot,\cdot\>_{\E}$ acquires the form
\begin{equation}
\<(X,\xi),(Y,\eta)\>_{\E} = (-1)^{(|X|+\ell)\ell} \xi(Y) + (-1)^{|X|(|Y|+\ell)} \eta(X),
\end{equation}
for all $(X,\xi),(Y,\eta) \in \Gamma_{\E}(M)$. It is not difficult to see that it is graded symmetric of degree $\ell + \ell'$. Finally, it is clear that $|[(X,\xi),(Y,\eta)]_{\E}| = |(X,\xi)| + |(Y,\eta)| + \ell + \ell'$. We conclude that all operations have the correct degree $\ell + \ell'$. The induced map $\dD: \C^{\infty}_{\M}(M) \rightarrow \Gamma_{\E}(M)$ reads
\begin{equation} \label{eq_dDdoubleLiebialgexpl}
\dD{f} = ( (-1)^{\ell} \da'^{T}(\dr{f}), (-1)^{(\ell+1)\ell'} \da^{T}(\dr{f})) = (-1)^{\ell + \ell'} ( (-1)^{|f|\ell'} \dr_{\A^{\ast}}f, (-1)^{(|f|+\ell')\ell} \dr_{\A}f). 
\end{equation}
It remains to prove the graded Courant algebroid axioms (see Definition \ref{def_gCourant}). Let us only sketch the proof and leave the details to the reader. 
\begin{enumerate}[(i)]
\item The compatibility (\ref{eq_CApairingcomp}) of $[\cdot,\cdot]_{\E}$ with $\<\cdot,\cdot\>_{\E}$ reduces to the definition of Lie derivatives $\Li{}^{\A}$ and $\Li{}^{\A^{\ast}}$. One also uses the fact that $(\dr_{\A}\xi)(X,Y) + (-1)^{\ell + |X||Y|} (\dr_{\A}\xi)(X,Y) = 0$ and its dual counterpart using $\dr_{\A^{\ast}}$. 
\item The proof of (almost) graded skew-symmetry (\ref{eq_CAskewsym}) uses (\ref{eq_dDdoubleLiebialgexpl}) together with Cartan relations (\ref{eq_LACartan1}, \ref{eq_LACartan1b}) and their dual counterparts. 
\item The most complicated bit is naturally the graded Jacobi identity (\ref{eq_CAgJI}). Its validity is equivalent to the graded Jacobi identities for both $\A$ and $\A^{\ast}$, together with the duality conditions (\ref{eq_GLBAduality01}, \ref{eq_GLBAduality11}) and their dual counterparts (they hold thanks to Proposition \ref{tvrz_gLBAdual}).
\end{enumerate}
The claim about $\A[\ell']$ and $\A^{\ast}[\ell]$ being Dirac structures is obvious. 
\end{proof}
\begin{rem}
One can see that $\E$, the anchor $\rho$, and the bracket $[\cdot,\cdot]_{\E}$ are symmetric with respect to $(\A,\ell) \leftrightarrow (\A^{\ast},\ell')$. However, for the pairing, one can write 
\begin{equation}
\<(X,\xi),(Y,\eta)\>_{\E} = (-1)^{\ell \ell'} \{ (-1)^{(|\xi|+\ell')\ell'} X(\eta) + (-1)^{|\xi|(|\eta|+\ell')} Y(\xi) \}.
\end{equation}
This shows that the double of the dual Lie bialgebroid $(\A^{\ast},\A)$ is $(\E, \rho, (-1)^{\ell \ell'} g_{\E}, [\cdot,\cdot]_{\E})$. 
\end{rem}
\begin{tvrz}
Let $(\A,\A^{\ast})$ be a graded Lie bialgebroid of bidegree $(\ell,\ell')$. Then for any $q \in \Z$, $(\A[q], \A[q]^{\ast})$ is a graded Lie bialgebroid of bidegree $(\ell+q, \ell'-q)$. 
\end{tvrz}
\begin{proof}
It is easy to see that $(\A[q],\da,[\cdot,\cdot]_{\A})$ forms a graded Lie algebroid of degree $\ell + q$. Since $\A[q]^{\ast} \cong \A^{\ast}[-q]$, we may view $\A[q]^{\ast}$ as a graded Lie algebroid of degree $\ell' - q$. Since
\begin{equation}
\E = \A[\ell'] \oplus \A^{\ast}[\ell] \cong (\A[q])[\ell'-q] \oplus (\A^{\ast}[-q])[\ell+q],
\end{equation}
we may compare the structures induced by both $(\A,\A^{\ast})$ and $(\A[q],\A[q]^{\ast})$ on $\E$ as in Theorem \ref{thm_doubleofgLBA} $(i)$-$(iv)$. It turns out that that the  anchors $\rho$ and the brackets $[\cdot,\cdot]_{\E}$ are completely the same, whereas the fiber-wise metrics $\<\cdot,\cdot\>_{\E}$ differ only by an overall sign. In particular, the graded Jacobi identity for $[\cdot,\cdot]_{\E}$ implies the duality condition (\ref{eq_GLBAduality_even}) or (\ref{eq_GLBAduality_odd}) for $(\A[q], \A[q]^{\ast})$. 
\end{proof}
\begin{example}
Consider the graded Lie algebroid $\A = T\M$ of degree $\ell = 0$, see Example \ref{ex_standardgLA}. 

One can view the dual $\A^{\ast} = T^{\ast}\M$ as a graded Lie algebroid of degree $\ell'$ with the trivial anchor of degree $\ell'$ and the trivial bracket of degree $\ell'$. This may seem strange, but since graded vector spaces are (in our viewpoint) just sequences of vector spaces, they have a zero vector in each degree. Consequently, we can have trivial graded linear maps of any given degree. Clearly, $(\A,\A^{\ast})$ is a graded Lie bialgebroid of bidegree $(0,\ell')$. Its double $\E = T[\ell']\M \oplus T^{\ast}\M$ is precisely the degree $\ell'$ graded Dorfman bracket  from Example \ref{ex_gDorfman} (for $H = 0$). 
\end{example}

\begin{example}
Let $(\M,\Pi)$ be a graded Poisson manifold of degree $\ell'$, see Example \ref{ex_gPoisson} and Example \ref{ex_PoissonasDirac} for details and notation. We assume $H = 0$. Observe that $(T^{\ast}\M, \Pi^{\sharp}, [\cdot,\cdot]_{\Pi})$ then becomes a graded Lie algebroid of degree $\ell'$. 

There is thus the corresponding Lie algebroid differential $\dr_{\Pi}$ acting on multivector fields. It turns out that it can be written simply as 
\begin{equation} \label{eq_dPiformula}
\dr_{\Pi} X = (-1)^{\ell'} [\Pi,X]_{S},
\end{equation}
for all $X \in \X^{p}_{\M}(M)$ (for even $\ell')$ or for all $X \in \~\X^{p}_{\M}(M)$ (for odd $\ell')$. Since both $\dr_{\Pi}$ and $(-1)^{\ell'}[\Pi,\cdot]_{S}$ satisfy the graded Leibniz rule (\ref{eq_dAgLeibniz}) or (\ref{eq_dAgLeibniz2}), it suffices to consider $p \in \{0,1\}$. The rest is a straightforward verification using the properties of $[\cdot,\cdot]_{S}$ and definitions. 

Now, $[\cdot,\cdot]_{S}$ is precisely the Schouten-Nijenhuis bracket associated to the graded Lie algebroid $(T\M, \1_{T\M}, [\cdot,\cdot])$ of degree $0$. It follows from (\ref{eq_dPiformula}) and the graded Jacobi identity (\ref{eq_LASNB_gJacobi}) or (\ref{eq_LASNB2_gJacobi}) that $\dr_{\Pi}$ is the graded derivation of $[\cdot,\cdot]_{S}$. In particular, this proves that $(T\M, \1_{T\M}, [\cdot,\cdot])$ and $(T^{\ast}\M, \Pi^{\sharp}, [\cdot,\cdot]_{\Pi})$ form a graded Lie bialgebroid $(T\M,T^{\ast}\M)$ of bidegree $(0,\ell')$. 
\end{example}

\begin{example}
Suppose $\M = \{ \ast \}$ is a singleton graded manifold. Then every graded Lie algebroid of degree $\ell$ reduces to a graded Lie algebra $(\g, [\cdot,\cdot]_{\g})$ of degree $\ell$. A graded Lie bialgebroid is in this case called a \textbf{graded Lie bialgebra}. Equivalently, it can be reformulated as follows. 

For any graded vector space $V$, the space of all graded linear endomorphisms $\Lin(V)$ forms a graded Lie algebra of degree zero. A graded linear map $\varrho: \g \rightarrow \Lin(V)$ of degree $|\varrho| = \ell$ is called a \textbf{representation of $\g$ on $V$}, if $\varrho([x,y]_{\g}) = [\varrho(x),\varrho(y)]$ for all $x,y \in \g$. Define 
\begin{equation}
(\ad^{\ast}_{x} \xi)(y) := - (-1)^{(|x|+\ell)|\xi|} \xi( [x,y]_{\g}), 
\end{equation}
for all $\xi \in \g^{\ast}$ and $x,y \in \g$. This defines the \textbf{coadjoint representation} of $\g$ on $\g^{\ast}$. For each $p \in \N_{0}$, let $\frm_{p}(\g^{\ast})$ denote the space of all graded $p$-linear maps from $\g^{\ast}$ to $\R$. There is an \textbf{adjoint representation } of $\g$ on $\frm_{p}(\g^{\ast})$ given for each $\omega \in \frm_{p}(\g^{\ast})$ and $x \in \g$ by  
\begin{equation}
[\ad_{x}^{(p)} \omega](\xi_{1},\dots,\xi_{p}) = -\sum_{r=1}^{p} (-1)^{(|x|+\ell)(|\omega|+ |\xi_{1}| + \dots + |\xi_{r-1}|)} \omega( \xi_{1}, \dots, \ad^{\ast}_{x} \xi_{r}, \dots, \xi_{p}), 
\end{equation}
for all $\xi_{1},\dots,\xi_{p} \in \g^{\ast}$. For any graded vector space $V$ and $p \in \N_{0}$, there is a a space $\frc^{p}(\g,V)$ of all graded $p$-linear maps $\omega: \g \times \dots \times \g \rightarrow V$ satisfying for each $i \in \{1,\dots,p-1\}$ the conditions
\begin{equation} \label{eq_cpgVdefined}
\omega(\dots,x_{i},x_{i+1},\dots) = -(-1)^{|x_{i}||x_{i+1}|} \omega(\dots, x_{i+1},x_{i}, \dots)
\end{equation}
and $|\omega(x_{1},\dots,x_{p})| = |\omega| + |x_{1}| + \dots + |x_{p}|$ for all $x_{1},\dots,x_{p} \in \g$. We declare $\frc^{0}(\g,V) = V$. Similarly, one defines a graded vector space $\~\frc^{p}(\g,V)$ of symmetric $p$-linear maps, where the sign on the right-hand side of (\ref{eq_cpgVdefined}) is opposite. Finally, for each $p \in \N_{0}$, there is a degree $\ell$ graded linear map $\Delta_{\g}: \frc^{p}(\g,V) \rightarrow \frc^{p+1}(\g,V)$ (for even $\ell$) and $\Delta_{\g}: \~\frc^{p}(\g,V) \rightarrow \~\frc^{p+1}(\g,V)$ (for odd $\ell$). It is defined using the formulas similar to (\ref{eq_dAeven}) of (\ref{eq_dAodd}), with $\da$ replaced by $\varrho$ and $[\cdot,\cdot]_{\A}$ by $[\cdot,\cdot]_{\g}$. It follows that $\Delta_{\g}^{2} = 0$. In this way, we obtain a cochain complex $(\frc^{\bullet}(\g,V), \Delta_{\g})$ (for even $\ell$) or $(\~\frc^{\bullet}(\g,V), \Delta_{\g})$ (for odd $\ell$) of graded vector spaces. 

Now, suppose that we have a graded Lie algebra structure $(\g^{\ast},[\cdot,\cdot]_{\g^{\ast}})$ of degree $\ell'$, and consider a degree $\ell'$ graded linear map $\delta: \g \rightarrow \frm_{2}(\g^{\ast})$ defined as 
\begin{equation}
[\delta(x)](\xi,\eta) := (-1)^{|x|(\ell+\ell'+1)} (\dr_{\g^{\ast}} x)(\xi,\eta) \equiv -(-1)^{|x|(\ell+\ell') + |\xi|\ell} x([\xi,\eta]_{\g^{\ast}}). 
\end{equation}
We can interpret $\delta$ as a $1$-cochain of degree $|\delta| = \ell'$ in $\frc^{1}(\g, \frm_{2}(\g^{\ast}))$ or $\~\frc^{1}(\g, \frm_{2}(\g^{\ast}))$, depending on the parity of $\ell$. One can with some effort (which we leave to the reader) arrive to the following statement: $(\g,\g^{\ast})$ is a graded Lie bialgebra of bidegree $(\ell,\ell')$, iff $\Delta_{\g}\delta = 0$. 

By Theorem \ref{thm_doubleofgLBA}, the graded vector space $\d := \g[\ell'] \oplus \g^{\ast}[\ell]$ can be then equipped with a graded Lie algebra bracket $[\cdot,\cdot]_{\d}$ of degree $\ell + \ell'$. Explicitly, it is given by
\begin{equation} \label{eq_bialgebradouble}
\begin{split}
[(x,\xi),(y,\eta)]_{\d} = \big( & [x,y]_{\g} + (-1)^{(|x|+\ell)(\ell+\ell')} \ad^{\ast}_{\xi}y - (-1)^{(|x|+\ell')(|y|+\ell)} \ad^{\ast}_{\eta}x, \\
& [\xi,\eta]_{\g^{\ast}} + (-1)^{(|x|+\ell)(\ell+\ell')} \ad^{\ast}_{x}\eta - (-1)^{(|x|+\ell')(|y|+\ell)} \ad^{\ast}_{y}\xi \big),
\end{split}
\end{equation}
for all $(x,\xi),(y,\eta) \in \d$. It follows that the graded Jacobi identities for $[\cdot,\cdot]_{\d}$ are equivalent to the graded Jacobi identities for both $\g$ and $\g^{\ast}$, together with the duality condition $\Delta_{\g}\delta = 0$. 

Let us finish with an actual example of a graded Lie bialgebra. Suppose $\g$ for some $q \in \Z - \{0\}$ satisfies $\dim(\g_{q}) = \dim(\g_{0}) = 1$ and $\dim(\g_{j}) = 0$ for $j \notin \{0,q\}$. Let $(t_{1},t_{2})$ be a total basis of $\g$ with $|t_{1}| = q$ and $|t_{2}| = 0$. Then
\begin{equation}
[t_{1},t_{2}]_{\g} = t_{1}, \; \; [t_{1},t_{1}]_{\g} = 0, \; \; [t_{2},t_{2}]_{\g} = 0
\end{equation}
makes $(\g,[\cdot,\cdot]_{\g})$ into a graded Lie algebra of degree $\ell = 0$. If $(t^{1},t^{2})$ is the dual total basis of $\g^{\ast}$, we have $|t^{1}| = -q$ and $|t^{2}| = 0$. It follows that by declaring 
\begin{equation}
[t^{1},t^{2}]_{\g^{\ast}} = t^{2}, \; \;  [t^{1},t^{1}]_{\g^{\ast}} = 0, \; \; [t^{2},t^{2}]_{\g^{\ast}} = 0,
\end{equation} 
we make $(\g^{\ast},[\cdot,\cdot]_{\g^{\ast}})$ into a graded Lie algebra of degree $\ell' = q$. The ``mixed'' brackets od $\d$ can be calculated from (\ref{eq_bialgebradouble}), giving the expressions
\begin{align}
[(t_{1},0),(0,t^{1})]_{\d} = & \ (0,-t^{2}), \\
[(t_{1},0),(0,t^{2})]_{\d} = & \ (0,0), \\
[(t_{2},0),(0,t^{2})]_{\d} = & \ (-t_{1},0), \\
[(t_{2},0),(0,t^{1})]_{\d} = & \ (t_{2},t^{1}).
\end{align} 
It is straightforward to verify that $(\d,[\cdot,\cdot]_{\d})$ constitutes a graded Lie algebra of degree $q$, whence $(\g,\g^{\ast})$ forms a graded Lie bialgebra of bidegree $(0,q)$. 
\end{example}

\section*{Acknowledgments}
First and foremost, I would like to express gratitude to my family for their patience and support. 

Next, I would like to thank Branislav Jurčo, Alexei Kotov and Vít Tuček for helpful discussions. 
The author is grateful for a financial support from MŠMT under grant no. RVO 14000. 

\bibliography{bib}
\appendix
\section{Structures induced by graded Lie algebroids} \label{sec_appendixA}
For the purposes of this section, let $(\A, \da, [\cdot,\cdot]_{\A})$ be a fixed Lie algebroid of degree $\ell$ over a graded manifold $\M$. To any graded vector bundle, one can contruct a graded vector space of its $p$-forms. 

\begin{definice} \label{def_pform}
Let $p \in \N_{0}$. By a \textbf{skew-symmetric $p$-form} $\omega$ on a graded vector bundle $\A$ of degree $|\omega|$, we mean a $p$-linear map 
\begin{equation}
\omega: \Gamma_{\A}(M) \times \dots \times \Gamma_{\A}(M) \rightarrow \C^{\infty}_{\M}(M)
\end{equation}
satisfying $|\omega(X_{1},\dots,X_{p})| = |\omega| + |X_{1}| + \dots + |X_{p}|$ for all $X_{1},\dots,X_{p} \in \Gamma_{\A}(M)$, such that 
\begin{equation}
\omega(f X_{1}, \dots, X_{p}) = (-1)^{|f||\omega|} f \omega(X_{1},\dots,X_{p}), 
\end{equation}
\begin{equation} \label{eq_pform_skewsym}
\omega(X_{1}, \dots, X_{i},X_{i+1},\dots,X_{p}) = - (-1)^{|X_{i}||X_{i+1}|} \omega(X_{1}, \dots, X_{i+1},X_{i},\dots,X_{p}),
\end{equation}
for all $X_{1},\dots,X_{p} \in \Gamma_{\A}(M)$, $f \in \C^{\infty}_{\M}(M)$ and $i \in \{1,\dots,p-1\}$. For each $p \in \N_{0}$, $p$-forms form a graded $\C^{\infty}_{\M}(M)$-module which we denote as $\Omega^{p}_{[\A]}(M)$. We declare $\Omega^{0}_{[\A]}(M) := \C^{\infty}_{\M}(M)$. 
\end{definice}
For each $U \in \Op(M)$, one can declare $\Omega^{p}_{[\A]}(U) := \Omega^{p}_{[\A|_{U}]}(U)$. Using partitions of unity, one can make the assignment $U \mapsto \Omega^{p}_{[\A]}(U)$ into a sheaf of graded $\C^{\infty}_{\M}$-modules, so that 
\begin{equation}
\omega|_{V}( X_{1}|_{V}, \dots, X_{p}|_{V}) = \omega(X_{1},\dots,X_{p})|_{V},
\end{equation}  
for all $V \subseteq U \subseteq M$, $\omega \in \Omega^{1}_{[\A]}(U)$ and $X_{1},\dots,X_{p} \in \Gamma_{\A}(U)$. In this way we obtain a \textbf{sheaf $\Omega^{p}_{[\A]}$ of skew-symmetric $p$-forms on $\A$}.
\begin{rem}
Similarly to \S 6.2 of \cite{vysoky2021global}, $p$-forms on $\A$ can be obtained as a certain subsheaf of $\C^{\infty}_{\A[1+s]}$. The two sheaves are naturally isomorphic, except that one has to use an alternative $\Z$-grading different from the one coming with the sheaf $\C^{\infty}_{\A[1+s]}$. See Proposition 6.4-$(ii)$ there. 
\end{rem}
\begin{rem}
It follows from Example \ref{ex_dual} and definitions that the sheaf $\Omega^{1}_{[\A]}$ is the sheaf of sections of the dual vector bundle $\A^{\ast}$. Similarly to the ordinary case, one can identify sections of $\A$ with the sections of the double dual $(\A^{\ast})^{\ast}$. However, in the graded setting, the action of $X \in \Gamma_{\A}(M)$ on $\omega \in \Omega^{1}_{[\A]}(M)$ is given by the formula containing an additional sign, that is 
\begin{equation} \label{eq_vfactionon1form}
X(\omega) := (-1)^{|X||\omega|} \omega(X). 
\end{equation}
This is in accordance with the Koszul sign convention and one should not miss this subtle difference. 
\end{rem}
The (exterior) product of forms can be most easily obtained from the graded algebra product of the sheaf $\C^{\infty}_{\A[1+s]}$. Since we will not actually use it in this paper, we only state the result. Note that the full exterior algebra $\Omega_{[\A]}(M)$ is \textit{not} a direct product $\prod_{p=0}^{\infty} \Omega^{p}_{[\A]}(M)$ with the $\Z$-grading we use here. For this reason, we will always talk exclusively about ``homogeneous'' forms. 

\begin{tvrz} \label{tvrz_extproduct}
There is an $\R$-bilinear product $\^: \Omega^{p}_{[\A]}(M) \times \Omega^{q}_{[\A]}(M) \rightarrow \Omega^{p+q}_{[\A]}(M)$ for each $p,q \in \N_{0}$. It has the following properties:
\begin{enumerate}[(i)]
\item It is compatible with the grading introduced in Definition \ref{def_pform}, that is $|\omega \^ \tau| = |\omega| + |\tau|$ for all $\omega \in \Omega^{p}_{[\A]}(M)$ and $\tau \in \Omega^{q}_{[\A]}(M)$.
\item One has $\omega \^ \tau = (-1)^{(p + |\omega|)(q + |\tau|)} \tau \^ \omega$ for all $\omega \in \Omega^{p}_{[\A]}(M)$ and $\tau \in \Omega^{q}_{[\A]}(M)$.
\item It is distributive with respect to the addition: $(\omega + \omega') \^ \tau = \omega \^ \tau + \omega' \^ \tau$ for all $\omega,\omega' \in \Omega^{p}_{[\A]}(M)$ with $|\omega| = |\omega'|$ and $\tau \in \Omega^{q}_{[\A]}(M)$. 
\item For all $f \in \C^{\infty}_{\M}(M) \equiv \Omega^{0}_{[\A]}(M)$ and all $\omega \in \Omega^{p}_{[\A]}(M)$, one has $f \^ \omega = f \omega$.  
\item It satisfies $\omega \^ (\tau \^ \eta) = (\omega \^ \tau) \^ \eta$ for all $\omega \in \Omega^{p}_{[\A]}(M)$, $\tau \in \Omega^{q}_{[\A]}(M)$ and $\eta \in \Omega^{r}_{[\A]}(M)$.
\item For each $U \in \Op(M)$, one has $(\omega \^ \tau)|_{U} = \omega|_{U} \^ \tau|_{U}$. 
\end{enumerate}
$\^$ is usually called the \textbf{exterior product}.
\end{tvrz}
Explicit formulas for $\omega \^ \tau$ evaluated on a $(p+q)$-tuple of vector fields are in fact quite complicated. Should one actually need them, they can be obtained using the interior product which we will now introduce. 

Similarly, one can define define the graded space $\~\Omega^{p}_{[\A]}(M)$ of \textbf{symmetric $p$-forms on $\A$}, where we replace (\ref{eq_pform_skewsym}) by 
\begin{equation}
\omega(X_{1},\dots,X_{i},X_{i+1},\dots,X_{p}) = (-1)^{|X_{i}||X_{i+1}|} \omega(X_{1},\dots,X_{i+1},X_{i},\dots,X_{p}),
\end{equation}
for a given $\omega \in \~\Omega^{p}_{[\A]}(M)$ and all $X_{1},\dots,X_{p} \in \Gamma_{\A}(M)$. 

There is a direct analogue of Proposition \ref{tvrz_extproduct} useful in order to get an $\R$-bilinear product $\odot: \~\Omega^{p}_{[\A]}(M) \times \~\Omega^{q}_{[\A]}(M) \rightarrow \~\Omega^{p+q}_{[\A]}(M)$, except that now
\begin{equation}
\omega \odot \tau = (-1)^{|\omega||\tau|} \tau \odot \omega,
\end{equation}
for all $\omega \in \~\Omega^{p}_{[\A]}(M)$ and $\tau \in \~\Omega^{q}_{[\A]}(M)$. 

Finally, one can define \textbf{skew-symmetric $p$-vector fields on $\A$} as $\X^{p}_{[\A]}(M) := \Omega^{p}_{[\A^{\ast}]}(M)$, and \textbf{symmetric $p$-vector fields on $\A$} as $\~\X^{p}_{[\A]}(M) := \~\Omega^{p}_{[\A^{\ast}]}(M)$. 

Now, every graded Lie algebroid of degree $\ell$ induces natural structures on its algebras of forms and multivector fields. This generalizes the similar story for ordinary Lie algebroids, see e.g. \S 7.1 of \cite{Mackenzie}. Let us only write down the resulting statements, the proofs are straightforward. The situation differs based on the parity of $\ell$. 
\begin{tvrz} \label{tvrz_dAeven}
Suppose $(\A,\da,[\cdot,\cdot]_{\A})$ is a graded Lie algebroid of \textbf{even} degree $\ell$. 

Then for each $p \in \N_{0}$, there exists an $\R$-linear map $\dr_{\A}: \Omega_{[\A]}^{p}(M) \rightarrow \Omega_{[\A]}^{p+1}(M)$ of degree $\ell$, defined for each $\omega \in \Omega^{p}_{[\A]}(M)$ and all $X_{1},\dots,X_{p+1} \in \Gamma_{\A}(M)$ by the formula
\begin{equation} \label{eq_dAeven}
\begin{split}
(\dr_{\A}\omega) & (X_{1},\dots,X_{p+1}) = \\
= & \ \sum_{r=1}^{p+1} (-1)^{r+1 + (|X_{r}|-1)|\omega| + |X_{r}|(|X_{1}| + \dots + |X_{r-1}|)} [\hat{\da}(X_{r})](\omega(X_{1},\dots,\hat{X}_{r},\dots,X_{p+1})) \\
& + \hspace{-5mm} \sum_{1 \leq i < j \leq p+1} \hspace{-3mm} (-1)^{|\omega|+i+j} \sigma_{ij}(|X_{1}|,\dots,|X_{p+1}|) \omega([X_{i},X_{j}]_{\A},X_{2},\dots,\hat{X}_{i},\dots,\hat{X}_{j},\dots,X_{p+1}),
\end{split}
\end{equation}
where $\hat{X}$ denotes an omitted element. 

For each tuple $(q_{1},\dots,q_{p+1})$ of integers, $\sigma_{ij}(q_{1},\dots,q_{p+1})$ is the unique sign obtained by reordering the term $v_{1} \dots v_{p+1}$ to $v_{i}v_{j} v_{2} \dots \hat{v}_{i} \dots \hat{v}_{j} \dots v_{p+1}$, where we assume that $v_{k}$'s are independent graded variables which commute in accordance with the rule $v_{k} v_{\ell} = (-1)^{q_{k}q_{\ell}} v_{\ell} v_{k}$. 

For any $\omega \in \Omega^{p}_{[\A]}(M)$ and $\tau \in \Omega^{q}_{[\A]}(M)$, it satisfies the graded Leibniz rule 
\begin{equation} \label{eq_dAgLeibniz}
\dr_{\A}(\omega \^ \tau) = \dr_{\A}\omega \^ \tau + (-1)^{|\omega|+p} \omega \^ \dr_{\A}\tau,
\end{equation}
and it squares to zero, that is $\dr_{\A}(\dr_{\A}\omega) = 0$ for every $\omega \in \Omega^{p}_{[\A]}(M)$. 
\end{tvrz}

\begin{tvrz}
Suppose $(\A,\da,[\cdot,\cdot]_{\A})$ is a graded Lie algebroid of \textbf{odd} degree $\ell$. 

Then for each $p \in \N_{0}$, there exists an $\R$-linear map $\dr_{\A}: \~\Omega_{[\A]}^{p}(M) \rightarrow \~\Omega_{[\A]}^{p+1}(M)$ of degree $\ell$, defined for each $\omega \in \~\Omega^{p}_{[\A]}(M)$ and all $X_{1},\dots,X_{p+1} \in \Gamma_{\A}(M)$ by the formula
\begin{equation} \label{eq_dAodd}
\begin{split}
(\dr_{\A}\omega) & (X_{1},\dots,X_{p+1}) = \\
= & \ \sum_{r=1}^{p+1} (-1)^{(|X_{r}|)(|\omega|+1) + |X_{r}|(|X_{1}| + \dots + |X_{r-1}|)} [\hat{\da}(X_{r})](\omega(X_{1},\dots,\hat{X}_{r},\dots,X_{p+1})) \\
& - \hspace{-5mm} \sum_{1 \leq i < j \leq p+1} \hspace{-3mm} (-1)^{|\omega|+|X_{i}|} \sigma_{ij}(|X_{1}|,\dots,|X_{p+1}|) \omega([X_{i},X_{j}]_{\A},X_{2},\dots,\hat{X}_{i},\dots,\hat{X}_{j},\dots,X_{p+1}),
\end{split}
\end{equation}
where the signs $\sigma_{ij}(|X_{1}|,\dots,|X_{p+1}|)$ are defined in the same way as in Proposition \ref{tvrz_dAeven}. For any $\omega \in \~\Omega^{p}_{[\A]}(M)$ and $\tau \in \~\Omega^{q}_{[\A]}(M)$, it satisfies the graded Leibniz rule 
\begin{equation} \label{eq_dAgLeibniz2}
\dr_{\A}(\omega \odot \tau) = \dr_{\A}\omega \odot \tau + (-1)^{|\omega|} \omega \odot \dr_{\A}\tau,
\end{equation}
and it squares to zero, that is $\dr_{\A}(\dr_{\A}\omega) = 0$ for every $\omega \in \~\Omega^{p}_{[\A]}(M)$. 
\end{tvrz}
In fact, in both cases, the condition $\dr_{\A}^{2} = 0$ is equivalent to the graded Jacobi identity (\ref{eq_LAgJI}). Moreover, one can introduce the remaining standard operations for forms, thus obtaining the complete sets of Cartan relations.

\begin{tvrz}
Let $(\A, \da, [\cdot,\cdot]_{\A})$ be a graded Lie algebroid of \textbf{even} degree $\ell$. 

Then for each $p \in \N$ and $X \in \Gamma_{\A}(M)$, there is an $\C^{\infty}_{\M}(M)$-linear operator $i_{X}: \Omega^{p}_{[\A]}(M) \rightarrow \Omega^{p-1}_{[\A]}(M)$ of degree $|X|$, defined for each $\omega \in \Omega^{p}_{[\A]}(M)$ and all $X_{1},\dots,X_{p-1} \in \Gamma_{\A}(M)$ as 
\begin{equation}
(i_{X}\omega)(X_{1},\dots,X_{p-1}) := (-1)^{|\omega|(|X|-1)} \omega(X,X_{1},\dots,X_{p-1}).
\end{equation}
For $\omega \in \Omega^{0}_{[\A]}(M)$, we declare $i_{X} \omega = 0$. It satisfies the formula
\begin{equation}
i_{X}(\omega \^ \tau) = (i_{X}\omega) \^ \tau + (-1)^{(|X|-1)(|\omega|+p)} \omega \^ (i_{X}\tau),
\end{equation}
for every $\omega \in \Omega^{p}_{[\A]}(M)$ and $\tau \in \Omega^{q}_{[\A]}(M)$. For any $X \in \Gamma_{\A}(M)$, let
\begin{equation} \label{eq_LACartan1}
\Li{X}^{\A} = i_{X} \circ \dr_{\A} + (-1)^{|X|} \dr_{\A} \circ i_{X}.
\end{equation}
Then $\Li{X}^{\A}: \Omega^{p}_{[\A]}(M) \rightarrow \Omega^{p}_{[\A]}(M)$ is $\R$-linear of degree $|X| + \ell$ and satisfies 
\begin{equation}
\Li{X}^{\A}(\omega \^ \tau) = (\Li{X}^{\A} \omega) \^ \tau + (-1)^{|X|(|\omega|+p)} \omega \^ (\Li{X}^{\A} \tau), 
\end{equation}
for all $\omega \in \Omega^{p}_{[\A]}(M)$ and $\tau \in \Omega^{q}_{[\A]}(M)$. These operations satisfy the following set of Cartan relations:
\begin{align}
\label{eq_LACartan2} \Li{[X,Y]_{\A}}^{\A} = & \ \Li{X}^{\A} \circ \Li{Y}^{\A} - (-1)^{|X||Y|} \Li{Y}^{\A} \circ \Li{X}^{\A}, \\
\label{eq_LACartan3} i_{[X,Y]_{\A}} = & \ \Li{X}^{\A} \circ i_{Y} - (-1)^{(|X|(|Y|-1)} i_{Y} \circ \Li{X}^{\A}, \\
\label{eq_LACartan4} 0 = & \ \Li{X}^{\A} \circ \dr_{\A} - (-1)^{|X|} \dr_{\A} \circ \Li{X}^{\A}, \\
\label{eq_LACartan5} 0 = & \ i_{X} \circ i_{Y} - (-1)^{(|X|-1)(|Y|-1)} i_{Y} \circ i_{X}.
\end{align}
\end{tvrz}

\begin{tvrz}
Let $(\A, \da, [\cdot,\cdot]_{\A})$ be a graded Lie algebroid of \textbf{odd} degree $\ell$. 

Then for each $p \in \N$ and $X \in \Gamma_{\A}(M)$, there is an $\C^{\infty}_{\M}(M)$-linear operator $j_{X}: \~\Omega^{p}_{[\A]}(M) \rightarrow \~\Omega^{p-1}_{[\A]}(M)$ of degree $|X|$, defined for each $\omega \in \~\Omega^{p}_{[\A]}(M)$ and all $X_{1},\dots,X_{p-1} \in \Gamma_{\A}(M)$ as 
\begin{equation}
(j_{X}\omega)(X_{1},\dots,X_{p-1}) := (-1)^{|\omega||X|} \omega(X,X_{1},\dots,X_{p-1}).
\end{equation}
For $\omega \in \~\Omega^{0}_{[\A]}(M)$, we declare $j_{X} \omega = 0$. It satisfies the formula
\begin{equation}
j_{X}(\omega \odot \tau) = (j_{X}\omega) \odot \tau + (-1)^{|X||\omega|} \omega \odot (j_{X}\tau),
\end{equation}
for every $\omega \in \~\Omega^{p}_{[\A]}(M)$ and $\tau \in \~\Omega^{q}_{[\A]}(M)$. For any $X \in \Gamma_{\A}(M)$, let
\begin{equation} \label{eq_LACartan1b}
\Li{X}^{\A} = j_{X} \circ \dr_{\A} - (-1)^{|X|} \dr_{\A} \circ j_{X}.
\end{equation}
Then $\Li{X}^{\A}: \~\Omega^{p}_{[\A]}(M) \rightarrow \~\Omega^{p}_{[\A]}(M)$ is $\R$-linear of degree $|X| + \ell$ and satisfies 
\begin{equation}
\Li{X}^{\A}(\omega \odot \tau) = (\Li{X}^{\A} \omega) \odot \tau + (-1)^{(|X|+1)|\omega|} \omega \odot (\Li{X}^{\A} \tau), 
\end{equation}
for all $\omega \in \~\Omega^{p}_{[\A]}(M)$ and $\tau \in \~\Omega^{q}_{[\A]}(M)$. These operations satisfy the following set of Cartan relations:
\begin{align}
\label{eq_LACartan2b} \Li{[X,Y]_{\A}}^{\A} = & \ \Li{X}^{\A} \circ \Li{Y}^{\A} - (-1)^{(|X|+1)(|Y|+1)} \Li{Y}^{\A} \circ \Li{X}^{\A}, \\
\label{eq_LACartan3b} j_{[X,Y]_{\A}} = & \ \Li{X}^{\A} \circ j_{Y} - (-1)^{(|X|+1)|Y|} j_{Y} \circ \Li{X}^{\A}, \\
\label{eq_LACartan4b} 0 = & \ \Li{X}^{\A} \circ \dr_{\A} + (-1)^{|X|} \dr_{\A} \circ \Li{X}^{\A}, \\
\label{eq_LACartan5b} 0 = & \ j_{X} \circ j_{Y} - (-1)^{|X||Y|} j_{Y} \circ j_{X}.
\end{align}
\end{tvrz}
Finally, it turns out that any graded Lie algebroid induces a Schouten-Nijenhuis bracket on \textbf{both} skew-symmetric and symmetric multivector fields. 
\begin{tvrz}
Let $(\A,\da,[\cdot,\cdot]_{\A})$ be a graded Lie algebroid of an \textbf{arbitrary} degree $\ell$.

Then for each $p \in \N_{0}$ and $X \in \Gamma_{\A}(M)$, there is an $\R$-linear operator $\Li{X}^{\A}: \X^{p}_{[\A]}(M) \rightarrow \X^{p}_{[\A]}(M)$ of degree $|X|+\ell$, defined for each $Y \in \X^{p}_{[\A]}(M)$ and $\xi_{1},\dots,\xi_{p} \in \Omega^{1}_{[\A]}(M)$ as 
\begin{equation}
\begin{split}
(\Li{X}^{\A}Y)(\xi_{1},\dots,\xi_{p}) := & \ [\hat{\da}(X)](Y(\xi_{1},\dots,\xi_{p})) \\
& - \sum_{r=1}^{p} (-1)^{(|X|+\ell)(|Y| + |\xi_{1}| + \dots + |\xi_{r-1}|)} Y(\xi_{1}, \dots, \Li{X}^{\A}\xi_{r}, \dots, \xi_{p}). 
\end{split}
\end{equation}
For each $Y \in \X^{p}_{[\A]}(M)$ and $Z \in \X^{q}_{[\A]}(M)$, it satisfies the relation
\begin{equation}
\Li{X}^{\A}(Y \^ Z) = (\Li{X}^{\A}Y) \^ Z + (-1)^{(|X|+\ell)(|Y| + p)} Y \^ (\Li{X}^{\A}Z),
\end{equation}
and for all $X \in \Gamma_{\A}(M)$ and $\xi \in \Omega^{1}_{[\A]}(M)$, there holds the Cartan relation 
\begin{equation}
\Li{X}^{\A} \circ i_{\xi} - (-1)^{(|X|+\ell)(|\xi|-1)} i_{\xi} \circ \Li{X}^{\A} = i_{\Li{X}^{\A}\xi}.
\end{equation}
\end{tvrz}
\begin{tvrz} \label{tvrz_LASNB}
For each $p,q \in \N_{0}$, there is an $\R$-bilinear bracket $[\cdot,\cdot]_{\A}: \X^{p}_{[\A]}(M) \times \X^{q}_{[\A]}(M) \rightarrow \X^{p+q-1}_{[\A]}(M)$ having the following properties:
\begin{enumerate}[(i)]
\item One has $|[X,Y]_{\A}| = |X| + |Y| + \ell$ for all $X \in \X_{[\A]}^{p}(M)$ and $Y \in \X_{[\A]}^{q}(M)$. 
\item For every $f \in \X^{0}_{[\A]}(M) \equiv \C^{\infty}_{\M}(M)$, $X \in \X^{1}_{[\A]}(M) \equiv \Gamma_{\A}(M)$ and $Y \in \X^{q}_{[\A]}(M)$, one has 
\begin{equation} \label{eq_LASNB_special}
[f,Y]_{\A} = (-1)^{|f|+1} i_{\dr_{\A}{f}} Y, \; \; [X,Y]_{\A} = \Li{X}^{\A}Y.
\end{equation}
\item For each $X \in \X^{p}_{[\A]}(M)$ and $Y \in \X^{q}_{\M}(M)$, one has
\begin{equation} \label{eq_LASNB_gsymmetry}
[X,Y]_{\A} = -(-1)^{(p + \ell + |X| - 1)(q + \ell + |Y| - 1)} [Y,X]_{\A}. 
\end{equation}
\item For all $X \in \X_{[\A]}^{p}(M)$, $Y \in \X_{[\A]}^{q}(M)$ and $Z \in \X_{[\A]}^{r}(M)$, one has
\begin{equation}
[X, Y \^ Z]_{\A} = [X,Y]_{\A} \^ Z + (-1)^{(p + \ell + |X| - 1)(q + \ell + |Y|)} Y \^ [X,Z]_{\A}.
\end{equation}
\item For all $X \in \X_{[\A]}^{p}(M)$, $Y \in \X_{[\A]}^{q}(M)$ and $Z \in \X_{[\A]}^{r}(M)$, one finds
\begin{equation} \label{eq_LASNB_gJacobi}
[X, [Y,Z]_{\A}]_{\A} = [[X,Y]_{\A}, Z]_{\A} + (-1)^{(p + \ell + |X| -1)(q + \ell + |Y| - 1)}[Y, [X,Z]_{\A}]_{\A}. 
\end{equation}
\end{enumerate}
\end{tvrz}
\begin{tvrz}
Let $(\A,\da,[\cdot,\cdot]_{\A})$ be a graded Lie algebroid of an \textbf{arbitrary} degree $\ell$.

Then for each $p \in \N_{0}$ and $X \in \Gamma_{\A}(M)$, there is an $\R$-linear operator $\Li{X}^{\A}: \~\X^{p}_{[\A]}(M) \rightarrow \~\X^{p}_{[\A]}(M)$ of degree $|X|+\ell$, defined for each $Y \in \~\X^{p}_{[\A]}(M)$ and $\xi_{1},\dots,\xi_{p} \in \Omega^{1}_{[\A]}(M)$ as 
\begin{equation}
\begin{split}
(\Li{X}^{\A}Y)(\xi_{1},\dots,\xi_{p}) := & \ [\hat{\da}(X)](Y(\xi_{1},\dots,\xi_{p})) \\
& - \sum_{r=1}^{p} (-1)^{(|X|+\ell)(|Y| + |\xi_{1}| + \dots + |\xi_{r-1}|)} Y(\xi_{1}, \dots, \Li{X}^{\A}\xi_{r}, \dots, \xi_{p}). 
\end{split}
\end{equation}
For each $Y \in \~\X^{p}_{[\A]}(M)$ and $Z \in \~\X^{q}_{[\A]}(M)$, it satisfies the relation
\begin{equation}
\Li{X}^{\A}(Y \odot Z) = (\Li{X}^{\A}Y) \odot Z + (-1)^{(|X|+\ell)(|Y|)} Y \odot (\Li{X}^{\A}Z),
\end{equation}
and for all $X \in \Gamma_{\A}(M)$ and $\xi \in \Omega^{1}_{[\A]}(M)$, there holds the Cartan relation 
\begin{equation}
\Li{X}^{\A} \circ j_{\xi} - (-1)^{(|X|+\ell)|\xi|} j_{\xi} \circ \Li{X}^{\A} = j_{\Li{X}^{\A}\xi}.
\end{equation}
\end{tvrz}
\begin{tvrz} \label{tvrz_LASNB2}
For each $p,q \in \N_{0}$, there is an $\R$-bilinear bracket $[\cdot,\cdot]_{\A}: \~\X^{p}_{[\A]}(M) \times \~\X^{q}_{[\A]}(M) \rightarrow \~\X^{p+q-1}_{[\A]}(M)$ having the following properties:
\begin{enumerate}[(i)]
\item One has $|[X,Y]_{\A}| = |X| + |Y| + \ell$ for all $X \in \~\X_{[\A]}^{p}(M)$ and $Y \in \~\X_{[\A]}^{q}(M)$. 
\item For every $f \in \~\X^{0}_{\A}(M) \equiv \C^{\infty}_{\M}(M)$, $X \in \~\X^{1}_{[\A]}(M) \equiv \Gamma_{\A}(M)$ and $Y \in \~\X^{q}_{[\A]}(M)$, one has 
\begin{equation} \label{eq_LASNB2_special}
[f,Y]_{\A} = (-1)^{|f|+1+\ell} j_{\dr_{\A}{f}} Y, \; \; [X,Y]_{\A} = \Li{X}^{\A}Y.
\end{equation}
\item For each $X \in \X^{p}_{[\A]}(M)$ and $Y \in \X^{q}_{[\A]}(M)$, one has
\begin{equation} \label{eq_LASNB2_gsymmetry}
[X,Y]_{\A} = -(-1)^{(|X|+\ell)(|Y|+\ell)} [Y,X]_{\A}. 
\end{equation}
\item For all $X \in \~\X_{[\A]}^{p}(M)$, $Y \in \~\X_{[\A]}^{q}(M)$ and $Z \in \~\X_{[\A]}^{r}(M)$, one has
\begin{equation}
[X, Y \odot Z]_{\A} = [X,Y]_{\A} \odot Z + (-1)^{(|X|+\ell)|Y|} Y \odot [X,Z]_{\A}.
\end{equation}
\item For all $X \in \~\X_{[\A]}^{p}(M)$, $Y \in \~\X_{[\A]}^{q}(M)$ and $Z \in \~\X_{[\A]}^{r}(M)$, one finds
\begin{equation} \label{eq_LASNB2_gJacobi}
[X, [Y,Z]_{\A}]_{\A} = [[X,Y]_{\A}, Z]_{\A} + (-1)^{(|X|+\ell)(|Y|+\ell)}[Y, [X,Z]_{\A}]_{\A}. 
\end{equation}
\end{enumerate}
\end{tvrz}

\end{document}